\documentclass[12pt]{article}
\usepackage[a4paper, total={17cm,25cm}]{geometry}

\usepackage{amsmath,amssymb,amsfonts}
\usepackage[alphabetic,y2k,initials]{amsrefs}

\usepackage{color, graphics}
\usepackage{float}
\usepackage{tikz}
\usepackage{pdfsync}
\usepackage{hyperref}

\usepackage{calrsfs}
\DeclareMathAlphabet{\pazocal}{OMS}{zplm}{m}{n}

\newcommand{\C}{\pazocal{C}}
\newcommand{\Q}{\pazocal{Q}}
\newcommand{\E}{\pazocal{E}}
\newcommand{\I}{\pazocal{I}}

\newcommand{\Curve}{\mathcal{C}}

\newcommand{\refeq}{\eqref}

\def\const{\mathrm{const}}

\def\max{\mathrm{max}}
\def\inv {\mathrm{inv}}

\def\sn{\mathrm{sn}}
\def\cn{\mathrm{cn}}
\def\dn{\mathrm{dn}}

\newtheorem{theorem}{Theorem}[section]
\newtheorem{proposition}[theorem]{Proposition}
\newtheorem{corollary}[theorem]{Corollary}
\newtheorem{lemma}[theorem]{Lemma}
\newtheorem{remark}[theorem]{Remark}
\newtheorem{example}[theorem]{Example}
\newtheorem{definition}[theorem]{Definition}

\newenvironment{proof}[1][Proof]{\noindent\textit{#1.} }{\hfill$\Box$\newline\medskip}

\numberwithin{equation}{section}

\usepackage{authblk}
\author[1,3]{Vladimir Dragovi\'c}
\author[2,3]{Milena Radnovi\'c}
\affil[1]{\textsc{The University of Texas at Dallas, Department for Mathematical Sciences}\newline\texttt{vladimir.dragovic@utdallas.edu}}
\affil[2]{\textsc{The University of Sydney, School of Mathematics and Statistics}\newline\texttt{milena.radnovic@sydney.edu.au}}
\affil[3]{\textsc{Mathematical Institute SANU, Belgrade}}

\date{}

\title{Caustics of Poncelet polygons and classical extremal polynomials}

\begin{document}

\maketitle

\begin{abstract}
A comprehensive analysis of periodic trajectories of billiards within ellipses  in the Euclidean plane is presented. The novelty of the approach is based on a relationship recently established by the authors between periodic billiard trajectories and extremal polynomials on the systems of $d$ intervals on the real line and ellipsoidal billiards in $d$-dimensional space. Even in the planar case, systematically studied in the present paper it leads to new results in characterizing $n$ periodic trajectories vs. so-called $n$ elliptic periodic trajectories, which are $n$-periodic in elliptical coordinates. The characterizations are done both in terms of the underlying elliptic curve and divisors on it and in terms of polynomial functional equations, like Pell's equation. This new approach also sheds light on some classical results. In particular we connect search for caustics which generate periodic trajectories with three classical classes of extremal polynomials on two intervals, introduced by Zolotarev and Akhiezer. The main classifying tool are winding numbers, for which we provide several interpretations, including one in terms of numbers of points of alternance of extremal polynomials. The latter implies important inequality between the winding numbers, which as a consequence, provides another proof of monotonicity of rotation numbers. A complete catalog of billiard trajectories with small periods is  provided for $n=3, 4, 5, 6$ along with an effective search for caustics. As a byproduct, an intriguing connection between Cayle type conditions and discriminantly separable polynomials has been observed for all those small periods.
\end{abstract}

MSC2010: 14H70, 41A10 (70H06, 37J35, 26C05)

Key words: Poncelet polygons, elliptical billiards, Cayley conditions, extremal polynomials, elliptic curves,  periodic trajectories, caustics, Pell's equations,  Chebyshev polynomials, Zolotarev polynomials, Akhiezer polynomials, discriminantly separable polynomials.

\tableofcontents

\section{Introduction}
In our recent paper \cite{DragRadn2018} we have developed a strong link between the theory of billiards within quadrics in $d$-dimensional space  and the theory of  extremal polynomials on the systems of $d$ intervals on the real line. Using this link, we proved some fundamental properties of the billiard dynamics and
paved a road to a comprehensive study of periodic trajectories of the billiards within ellipsoids in the $d$-dimensional Euclidean space.

The goal of the present paper is to provide the case study of the basic, planar case. It is well known that a geometric manifestation of integrability of elliptical billiards is the existence of a caustic, a conic confocal with the boundary of a billiard table, which is tangent to every segment of a given billiard trajectory. It is also well known that the elliptical billiard dynamics is equivalent to projective-geometry situation considered by Poncelet in 1813/14, when the boundary conic and the caustic conic are in arbitrary position, not necessarily confocal. It should also be mentioned that it is possible to redefine the billiard reflection in projective-geometric terms of harmonic conjugation of four lines in a pencil, to associate a generalized billiard system to any pair of conics, a boundary and a caustic. This is the reason why we call periodic trajectories also the Poncelet polygons, or if the period is $n$, the Poncelet $n$-gons. The Poncelet theorem (\cite{Poncelet1822}, see also \cites{LebCONIQUES, GrifHar1978,BergerGeometryII, DragRadn2011book}) states that if such a polygon exists for the fixed boundary and the caustic, then there are infinitely many such polygons, sharing the boundary and the caustic. Cayley around 1853 derived a criterion which answers the question for two given conics, a boundary and a caustic, whether the dynamics they generate is periodic with a given period or not. Cayley
answered the question by translating it to an equivalent question whether  a given point of an elliptic curve is of a given order in the group structure defined by the elliptic curve. Some 35 years later, Halphen established a relationship between the Poncelet polygons and continued fractions and approximation theory, see \cite{Hal1888}*{Part 2, page 600}.

Using the fact that the classical Cayley condition can be reformulated as a polynomial functional equation of Pell's type, we are going to show  that search for caustics which generate periodic trajectories of a given period is intimately related to the classical extremal polynomials, namely to the Zolotarev polynomials and Akhiezer polynomials. Zolotarev was a talented student of Chebyshev, the founding father of a famous Sankt Petersburg school in the second half of the XIX century. Chebyshev introduced celebrated Chebyshev polynomials in his study of extremal problems on an interval. It is interesting to point out that Chebyshev came to the extremal questions from the engineering problems from the theory of mechanisms. These problems were intimately related to the main technological questions brought by the industrial revolution. Similar questions were also studied by Poncelet. Akhiezer, one of the prominent figures of XX century mathematics, significantly developed further the ideas of Chebyshev and his school. Based upon his deep results on orthogonal polynomials and their continuous analogues, the notion of Baker-Akhiezer functions emerged in work of I. M. Krichever forty year ago. Soon it became one of the main tools of modern algebro-geometric theory of integrable systems, established by Novikov, Dubrovin, and others  in 1970's, see \cite{DKN} and references therein.

Apart from the connections between the Cayley condition and Zolotarev and Akhiezer extremal polynomials,
we also establish in this paper an unexpected relationship between the Cayley condition and another class of polynomials,
so-called discrimintly separable polynomials. This class has been introduced quite recently, less than a decade ago, in \cite{Drag2010}, related to the celebrated Kowalevski top and the Kowalevski integration procedure for equations of motion of the top \cite{Kow1889}.
Later, such polynomials were related to other continuous and discrete integrable systems, see  \cites{DragKuk2014jgm,DragKuk2014rcd,DragKuk2014steklov, DragKuk2017}.

\subsubsection*{Organization of this paper}
The next Section \ref{sec:periodic} introduces the basic notions related to elliptical billiards, derives an algebro-geometric criterion for periodicity with a given period $n$ for billiard trajectories within given ellipse and a confocal caustic in terms of an elliptic curve  which is isomorphic, but different then so called Cayley's cubic, originally studied by Cayley. Then a detailed study and two characterizations of $n$ elliptic periodic trajectories are provided. Let us recall that a $n$ elliptic  periodic trajectory has the period $n$ in elliptic coordinates, but is not necessarily periodic in Cartesian coordinates. One of the characterizations is given in terms of the underlying elliptic curve. The second one is done in Section \ref{sec:poly} in terms of extremal polynomials on two intervals and solutions of polynomial functional equations, like the Pell equation. This Section concludes with a detailed analysis
of winding numbers and their different appearances. The derived properties of winding numbers serve as a main classifying tool
for the caustics which generate periodic trajectories of given period. They also lead to another proof of the monotonicity of rotation numbers. Section \ref{sec:examples} employs
the derived criteria for periodicity and winding numbers to study in detail the cases of small periods $n=2, 3, 4, 5, 6$. Section \ref{sec:classicalextrema} related the results from the previous Section to
three classes of classical extremal polynomials on two intervals introduced by Zolotarev and Akhiezer.
The last Section \ref{sec:separable} collects come intriguing, yet not fully understood, experimental observations that Cayley type conditions produce so called discriminantly seaparable polynomials.

\section{Elliptical billiards and periodic trajectories}\label{sec:periodic}
\subsection{Elliptical billiard and confocal families}
\label{sec:confocal}

\emph{Mathematical billiard} within a plane domain is a dynamical system where a particle moves without constraints within the domain, and obeys the billiard reflection law on the boundary \cite{KozTrBIL}, see Figure \ref{fig:reflection}.
Billiard trajectories are polygonal lines with verteces on the boundary.
\begin{figure}[h]
	\begin{center}
%poluose elipse
\pgfmathsetmacro{\a}{3}
\pgfmathsetmacro{\b}{2}

%ugaona koordinata tacke udara
\pgfmathsetmacro{\phi}{-30}

%koordinate tacke udara
\pgfmathsetmacro{\xc}{\a * cos(\phi)}
\pgfmathsetmacro{\yc}{\b * sin(\phi)}

%duzina tangente u tacki udara
\pgfmathsetmacro{\t}{5}

%komponente tangentnog vektora u tacki udara
\pgfmathsetmacro{\tx}{-\t*sin(\phi)/ \b}
\pgfmathsetmacro{\ty}{\t*cos(\phi)/ \a}
\begin{tikzpicture}

\filldraw[draw=gray!50,fill=gray!50] (\xc,\yc) -- ({\xc -\tx/2},{\yc-\ty/2}) arc ({atan(\ty / \tx)}:5:{-sqrt(\tx*\tx+\ty*\ty)/2}) -- (\xc,\yc);

\filldraw[draw=gray!50,fill=gray!50] (\xc,\yc) -- ({\xc +\tx/2},{\yc+\ty/2}) arc ({atan(\ty / \tx)}:85:{sqrt(\tx*\tx+\ty*\ty)/2}) -- (\xc,\yc);

\draw[thick](0,0) circle [x radius=\a, y radius=\b];

\draw[thick]({\xc -\tx }, {\yc-\ty}) to ({\xc +\tx }, {\yc+\ty});

%\draw[blue]({\xc},{\yc}) circle [radius=0.1];

\draw[very thick, ->] (\xc, \yc) to (0.5,-1.2);
\draw[very thick] (2.7,0.3) to (\xc, \yc) ;

\end{tikzpicture}

		\caption{The billiard reflection law: the impact and reflection angles are congruent.}\label{fig:reflection}
	\end{center}
\end{figure}
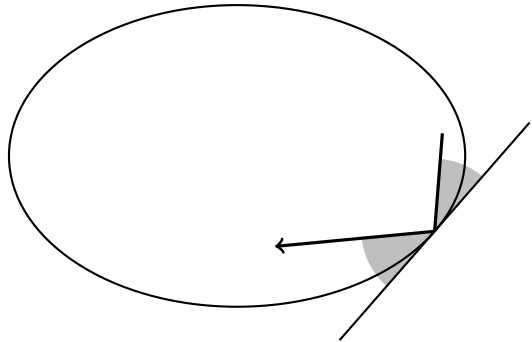

Mathematical billiard is an idealised model, where the billiard ball is replaced by a material point, and the friction and spin are neglected.
Such billiard system in the Euclidean space is a good model for the motion of light rays, with mirror boundary.
The dynamics has two different regimes: inside the billiard domain, and the impacts.
We assume that the impacts are \emph{absolutely elastic}, which means that the geometric billiard law is satisfied, i.e.~the impact and reflection angles are congruent to each other, and the speed remains unchanged.
Here, we assume that the material point is travelling under inertia between the impacts, although motion in a force field can also be considered, for example in a gravitational field \cite{KL1991}, with Hooke's potential \cites{Fed2001,Radn2015}, or with integrable potentials \cite{Drag2002}.

In this work, we will study \emph{elliptical billiards} -- the mathematical billiards within an ellipse in the Euclidean plane:
$$
\E\ :\ \frac{\mathsf{x}^2}{a}+\frac{\mathsf{y}^2}{b}=1,
\quad
a>b>0.
$$
Each trajectory of the elliptic billiard has \emph{a caustic}: a curve such that each segment of the trajectory lies on its tangent line.
Moreover, the caustics of billiard trajectories within $\E$ belong to the family of conics confocal with the boundary:
\begin{equation}\label{eq:confocal}
\C_{\lambda}\ :\ \frac {\mathsf{x}^2}{a-\lambda}+ \frac{\mathsf{y}^2}{b-\lambda }=1.
\end{equation}

We notice that the family \eqref{eq:confocal} contains two types of smooth conics: ellipses, corresponding to $\lambda<b$, and hyperbolas, corresponding to $\lambda\in(b,a)$, see Figure \ref{fig:confocal}.
\begin{figure}[h]
	\begin{center}
\begin{tikzpicture}

\draw[thick](0,0) circle [x radius={sqrt(8)}, y radius={sqrt(5)}];
\draw[thick](0,0) circle [x radius=2, y radius=1];
\draw[thick](0,0) circle [x radius={sqrt(6)}, y radius={sqrt(3)}];

\draw[black,fill=gray]({sqrt(3)},0) circle [x radius=0.1,y radius=0.1];
\draw[black,fill=gray]({-sqrt(3)},0) circle [x radius=0.1,y radius=0.1];

\draw[domain=-2:2,smooth,thick,variable=\t] plot ({sqrt(2*(\t*\t+1))},{\t});
\draw[domain=-2:2,smooth,thick,variable=\t] plot ({-sqrt(2*(\t*\t+1))},{\t});

\draw[domain=-2.5:2.5,smooth,thick,variable=\t] plot ({sqrt(1+\t*\t/2))},{\t});
\draw[domain=-2.5:2.5,smooth,thick,variable=\t] plot ({-sqrt(1+\t*\t/2))},{\t});

\draw[domain=-2.8:2.8,smooth,thick,variable=\t] plot ({0.3*sqrt(1+\t*\t/2.7))},{\t});
\draw[domain=-2.8:2.8,smooth,thick,variable=\t] plot ({-0.3*sqrt(1+\t*\t/2.7))},{\t});

\end{tikzpicture}
		\caption{A family of confocal conics in the plane.}\label{fig:confocal}
	\end{center}
\end{figure}
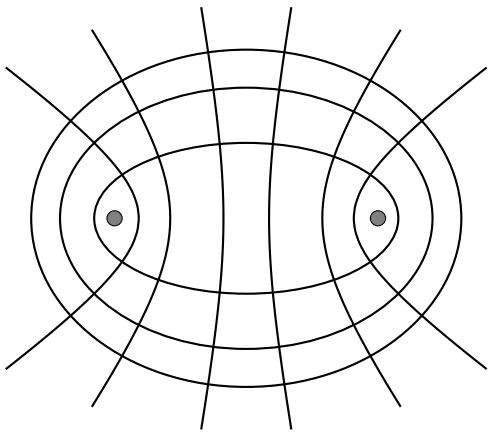

In addition, there are two degenerated conics and the family: the $\mathsf{x}$-axis, corresponding to $\lambda=b$; and the $\mathsf{y}$-axis, corresponding to $\lambda=a$.

Each point in the plane, which is not a focus of the confocal family, lies on exactly two conics $\C_{\lambda_1}$ and $\C_{\lambda_2}$ from \eqref{eq:confocal} -- one ellipse and one hyperbola, which are orthogonal to each other at the intersection point.
In other words, we can join to such a point a unique pair of parameters $(\lambda_1,\lambda_2)$, $\lambda_1<\lambda_2$, corresponding to the two confocal conics which contain the point.
The pair $(\lambda_1,\lambda_2)$ is called \emph{elliptic coordinates} of the point.
We note that points symmetric with respect to the $\mathsf{x}$- and $\mathsf{y}$-axes have the same elliptic coordinates.

All segments of the billiard trajectories within $\E$ with a fixed caustic $\C_{\lambda_0}$ lie in a domain $\Omega_{\lambda_0}$, that is:
\begin{itemize}
	\item if $\C_{\lambda_0}$ is an ellipse, that is $\lambda_0\in(0,b)$, $\Omega_{\lambda_0}$ is the annulus between the billiard boundary $\E$ and the caustic;
	\item if $\C_{\lambda_0}$ is a hyperbola, that is $\lambda_0\in(b,a)$, $\Omega_{\lambda_0}$ is the part within the billiard boundary $\E$ which is between the branches of the caustic.
\end{itemize}
In the elliptic coordinates, $\Omega_{\lambda_0}$ is given by:
$$
(\lambda_1,\lambda_2)\in[0,\alpha_1]\times[\alpha_2,a],
\quad
\alpha_1=\min\{b,\lambda_0\},
\
\alpha_2=\max\{b,\lambda_0\}.
$$
On any billiard trajectory, the value $\lambda_1=0$ is achieved at the reflection points on the boundary, value $\lambda_2=a$ at the intersection points with the $\mathsf{y}$-axis, while corresponding elliptic coordinate has value $\lambda_0$ at the touching points with the caustic, and value $b$ at the intersection points with $\mathsf{x}$-axis.
Between these points, the elliptic coordinates change monotonously.

For a periodic billiard trajectory, we introduce \emph{the winding numbers} $(m_0,m_1)$: $m_0$ is the number of its reflection points, and $m_1$ the number of its intersection with the $\mathsf{y}$-axis. The number $m_1$ is always even, since the $\mathsf{y}$-axis must be crossed even number of times along the period.

\subsection{Periodic trajectories and Cayley's cubic}\label{sec:cayley}

It is of particular interest to condsider periodic billiard trajectories -- the trajectories that become closed after certain number of reflections. In the next Lemma, we note that the type of the caustic may be determined by the period of a trajectory.

\begin{lemma}\label{lemma:hyperbola} The period of a closed billiard trajectory with a hyperbola as caustic is always even.
\end{lemma}
\begin{proof}
	Denote by $F_1, F_2$ the focal points of the boundary ellipse. Then every segment of the billiard trajectory intersects the segment $F_1F_2$. Thus, for a periodic trajectory there should be an even number of intersections of the trajectory with $F_1F_2$ and the period is even.
\end{proof}
The Poncelet theorem \cites{Poncelet1822,LebCONIQUES,DragRadn2011book} implies that all trajectories sharing the same caustic with a periodic elliptical billiard trajectory are also periodic, and moreover all these trajectories have the same period.
It is natural to ask about an analytic condition that would determine if elliptical billiard trajectories with a given caustic will become closed after certain number of bounces.
Such condition was derived by Cayley in the mid XIXth century.

Here, in Theorems \ref{th:curve-billiard} and \ref{th:cayley-billiard}, we present the derivation of the analytic conditions following the ideas Jacobi and Darboux \cites{JacobiGW,DarbouxSUR}, see also \cite{DragRadn2004}.
In Theorem \ref{th:cayley}, we will present the classical Cayley's conditions, since it has a slightly different form, and show they are equivalent to the ones presented in Theorem \ref{th:cayley-billiard}.

\begin{theorem}\label{th:curve-billiard}
	The billiard trajectories within $\E$ with caustic $\C_{\lambda_0}$ are $n$-periodic if and only if  $nQ_{0}\sim nQ_{\lambda_0}$ on the elliptic curve:
	\begin{equation}\label{eq:billiard-cubic}
	\Curve\ :\ y^2=(a-x)(b-x)(\lambda_0-x),
	\end{equation}
	with $Q_0$ being a point of $\Curve$ corresponding to $x=0$, and $Q_{\lambda_0}$ the point corresponding to $x=\lambda_0$.
\end{theorem}
\begin{proof}
	Consider the integral
	\begin{equation}\label{eq:integral}
	\frac{d\lambda_1}{\sqrt{(a-\lambda_1)(b-\lambda_1)(\lambda_0-\lambda_1)}}
	+
	\frac{d\lambda_2}{\sqrt{(a-\lambda_2)(b-\lambda_2)(\lambda_0-\lambda_2)}}
	\end{equation}
	along the polygonal line, which represents $n$ consecutive segments of a billiard trajectory within $\E$ with caustic $\C_{\lambda_0}$.
	That integral equals to zero along each line touching $\C_{\lambda_0}$.
	Thus, considering the behaviour of the elliptic coordinates along the trajectory with winding numbers $(m_0,m_1)$, $m_0=n$, we get that the first vertex of the polygonal line will coincide with the last one if and only if
	$$
	n(Q_0-Q_{\alpha_1})+m_1(Q_{\alpha_2}-Q_a)\sim0
	$$
	on $\Curve$.
	Here, we denoted by $Q_{\alpha_1}$, $Q_{\alpha_2}$, $Q_{a}$ the points of $\Curve$ corresponding to $x$ equal to $\alpha_1$, $\alpha_2$, $a$ respectively.
	Since $2Q_{\alpha_2}\sim2Q_a$ and $m_1$ is even, we obtained that the periodicity condition reduced to:
	$$
	n(Q_0-Q_{\alpha_1})\sim0
	$$
	Now, if $n$ is even, $2\Q_{\alpha_1}\sim2Q_{\lambda_0}$ implies $n(Q_0-Q_{\lambda_0})\sim0$.
	If $n$ is odd, then, according to Lemma \ref{lemma:hyperbola}, the caustic of the trajectory must be an ellipse, so $\alpha_1=\lambda_0$.
\end{proof}	

\begin{theorem}\label{th:cayley-billiard}
	The billiard trajectories within $\E$ with caustic $\C_{\lambda_0}$ are $n$-periodic if and only if:
	\begin{gather*}
	C_2=0,
	\quad
	\left|
	\begin{array}{cc}
	C_2 & C_3
	\\
	C_3 & C_4
	\end{array}
	\right|=0,
	\quad
	\left|
	\begin{array}{ccc}
	C_2 & C_3 & C_4
	\\
	C_3 & C_4 & C_5
	\\
	C_4 & C_5 & C_6
	\end{array}
	\right|=0,
	\dots
	\quad\text{for}\quad n=3,5,7,\dots
	\\
	B_3=0,
	\quad
	\left|
	\begin{array}{cc}
	B_3 & B_4
	\\
	B_4 & B_5
	\end{array}
	\right|=0,
	\quad
	\left|
	\begin{array}{ccc}
	B_3 & B_4 & B_5
	\\
	B_4 & B_5 & B_6
	\\
	B_5 & B_6 & B_7
	\end{array}
	\right|=0,
	\dots
	\quad\text{for}\quad n=4,6,8,\dots.
	\end{gather*}
	Here, we denoted:
	\begin{gather*}
	\sqrt{(a-x)(b-x)(\lambda_0-x)}=B_0+B_1x+B_2x^2+\dots,
	\\
	\frac{\sqrt{(a-x)(b-x)(\lambda_0-x)}}{\lambda_0-x}=C_0+C_1x+C_2x^2+\dots,
	\end{gather*}
	the Taylor expansions around $x=0$.
\end{theorem}
\begin{proof}
	Denote by $Q_{\infty}$ the point of $\Curve$ \refeq{eq:billiard-cubic} corresponding to $x=\infty$ and notice that
	\begin{equation}\label{eq:2Q}
	2Q_{\lambda_0}\sim2 Q_{\infty}.
	\end{equation}
	
	Consider first $n$ even.
	Because of \refeq{eq:2Q}, the condition $nQ_{0}\sim nQ_{\lambda_0}$ is equivalent to $nQ_{0}\sim nQ_{\infty}$, which is equivalent to the existence of a meromorphic function of $\Curve$ with the unique pole at $Q_{\infty}$ and unique zero at $Q_{0}$, such that the pole and the zero are both of the multiplicity $n$.
	The basis of $\mathcal{L}(nQ_{\infty})$ is:
	\begin{equation}\label{eq:basis-even}
	1,x,x^2,\dots,x^{n/2},y,xy, x^{n/2-2}y,
	\end{equation}
	thus a non-trivial linear combination of those functions with a zero of order $n$ at $x=0$ exists if and only if:
	$$
	\left|
	\begin{array}{llll}
	B_{n/2+1} & B_{n/2} & \dots & B_3\\
	B_{n/2+2} & B_{n/2+1} &\dots & B_4\\
	\dots\\
	B_{n-1} & B_n &\dots & B_{n/2+1}
	\end{array}
	\right|=0.
	$$
	
	Now, suppose $n$ is odd.
	Because of \refeq{eq:2Q}, the condition $nQ_{0}\sim nQ_{\lambda_0}$ is equivalent to $nQ_{0}\sim (n-1)Q_{\infty}+Q_{\lambda_0}$, which is equivalent to the existence of a meromorphic function of $\Curve$ with only two poles: of order $n-1$ at $Q_{\infty}$ and a simple pole at $Q_{\lambda_0}$, and unique zero at $Q_{0}$.
	The basis $\mathcal{L}( (n-1)Q_{\infty}+Q_{\lambda_0})$ is:
	\begin{equation}\label{eq:basis-odd}
	1,x,x^2,\dots,x^{(n-1)/2},\frac{y}{\lambda_0-x}, \frac{xy}{\lambda_0-x}, \dots, \frac{x^{(n-1)/2-1}y}{\lambda_0-x},
	\end{equation}
	thus a non-trivial linear combination of those functions with a zero of order $n$ at $x=0$ exists if and only if:
	$$
	\left|
	\begin{array}{llll}
	C_{(n-1)/2+1} & C_{(n-1)/2} & \dots & C_2\\
	C_{(n-1)/2+2} & C_{(n-1)/2+1} &\dots & C_3\\
	\dots\\
	C_{n-1} & C_n &\dots & C_{(n-1)/2+1}
	\end{array}
	\right|=0.
	$$
\end{proof}

We can rewrite the equation (\ref{eq:confocal}) of the family of confocal conics in a matrix form:
$$
(x\, y\, 1) M_ {\lambda}(x\, y \,1)^T=0,
$$
where
$$
M_{\lambda}=\left(
\begin{array}{ccc}
\dfrac1{a-\lambda} & 0 & 0\\
0 & \dfrac1{b-\lambda} & 0\\
0& 0& -1
\end{array}
\right).
$$
Following Cayley, we can use the matrices $M_{\lambda}$ to define another elliptic curve, so called
\emph{Cayley's cubic}
\begin{equation}\label{eq:cayley-cubic}
\Curve^*\ :\ y^2=\det(M_0+xM_{\lambda_0}).
\end{equation}
The curve \refeq{eq:cayley-cubic} can be explicitly written as:
$$
y^2=-(x+1) \left(\frac{1}{a-\lambda_0 }+\frac{x}{a}\right) \left(\frac{1}{b-\lambda_0 }+\frac{x}{b}\right).
$$
This is an elliptic curve with branching points corresponding to:
$$
x\in\left\{-1,\frac{a}{\lambda_0-a},\frac{b}{\lambda_0-b},\infty\right\}.
$$
Using the bilinear transformation:
\begin{gather*}
(x,y)\mapsto(x_1,y_1),
\\
x_1=\frac{\lambda_0 x}{1+x},
\quad
y_1=\const\cdot\frac{y}{1+x},
\end{gather*}
The branching points are transformed to $\{\infty,a,b,\lambda_0\}$ respectively and the Cayley's cubic into
$y_1^2=(a-x_1)(b-x_1)(\lambda_0-x_1)$.

Thus, $nP_0\sim nP_{\infty}$ on $\Curve^*$ is equivalent to $nQ_0\sim nQ_{\lambda_0}$ on $\Curve$.
We conclude that previous Theorem is equivalent to the following formulation of the classical Cayley condition.
\begin{theorem}[Cayley's condion,  \cites{Cayley1853,GrifHar1977,GrifHar1978,LebCONIQUES}]\label{th:cayley}
	There is a clo\-sed po\-ly\-go\-nal line with $n$ vertices, inscribed in $\E$ and circumscribed about $\C_{\lambda_0}$ if and only if $nP_{0}\sim nP_{\infty}$ on the Cayley's cubic $\Curve^*$, given by \refeq{eq:cayley-cubic}, with $P_0$ being one of the points of the curve $\Curve^*$ corresponding to $x=0$, and $P_{\infty}$ the point corresponding to $x=\infty$.
	
	Further, $nP_{0}\sim nP_{\infty}$ is equivalent to:
	\begin{gather*}
	A_2=0,
	\quad
	\left|
	\begin{array}{cc}
	A_2 & A_3
	\\
	A_3 & A_4
	\end{array}
	\right|=0,
	\quad
	\left|
	\begin{array}{ccc}
	A_2 & A_3 & A_4
	\\
	A_3 & A_4 & A_5
	\\
	A_4 & A_5 & A_6
	\end{array}
	\right|=0,
	\dots
	\quad\text{for}\quad n=3,5,7,\dots
	\\
	A_3=0,
	\quad
	\left|
	\begin{array}{cc}
	A_3 & A_4
	\\
	A_4 & A_5
	\end{array}
	\right|=0,
	\quad
	\left|
	\begin{array}{ccc}
	A_3 & A_4 & A_5
	\\
	A_4 & A_5 & A_6
	\\
	A_5 & A_6 & A_7
	\end{array}
	\right|=0,
	\dots
	\quad\text{for}\quad n=4,6,8,\dots,
	\end{gather*}
	with
	$$
	\sqrt{\det(M_0+xM_{\lambda_0})}=A_0+A_1x+A_2x^2+A_3x^3+\dots,
	$$
	being the Taylor expansion about $x=0$.
\end{theorem}

\subsection{Elliptic periodic trajectories}\label{sec:elliptic-periodic}

Points of the plane which are symmetric with respect to the coordinate axes share the same elliptic coordinates, thus there is no bijection between the elliptic and the Cartesian coordinates. Thus, we introduce a separate notion of periodicity in elliptic coordinates.

\begin{definition}
	A billiard trajectory is \emph{$n$-elliptic periodic} is it is $n$-periodic in elliptic coordinates joined to the confocal family \refeq{eq:confocal}.
\end{definition}

\begin{theorem}\label{th:even-periodic}
	A billiard trajectory within $\E$ is $2n$-periodic if and only if it is $n$-elliptic periodic.
\end{theorem}
\begin{proof}
	The symmetry implies that each $n$-elliptic periodic trajectory is symmetric with respect to one of the axes or the origin. Consequently, such a trajectory is always $2n$-periodic.
	
	%As in the proof of Theorem \ref{th:elliptic-periodic}, we denote $\{\alpha_1,\alpha_2\}=\{b,\lambda_0\}$, $\alpha_1<\alpha_2$.
	Along any $2n$-periodic trajectory, the elliptic coordinate $\lambda_1$  will trace the segment $[0,\alpha_1]$ $2n$ times in each direction.
	Similarly, $\lambda_2$ will trace the segment $[\alpha_2,a]$ an even number of times in each direction, since each periodic trajectory intersects the $\mathsf{y}$-axis even number of times.
	Each of these intersections corresponds to $\lambda_2 = a$.
	This all together implies that the given trajectory  will be elliptic periodic with period $n$.
\end{proof}

Thus, we see that in the Euclidean plane, there are no $2n$-periodic trajectories which are not $n$-elliptic periodic. The situation is different in higher-dimensional spaces, as is shown in \cite{DragRadn2018}.

Now, we will derive algebro-geometric conditions for elliptic periodic trajectories.

\begin{theorem}\label{th:elliptic-periodic}
	A billiard trajectory within $\E$ with the caustic $\Q_{\lambda_0}$ is $n$-elliptic periodic without being $n$-periodic if and only if one of the following conditions is satisfied on $\Curve$:
	\begin{itemize}
		\item[(a)] $\Q_{\lambda_0}$ is an ellipse and $n(Q_{0}-Q_{\lambda_0})+Q_b-Q_a\sim0$;
		\item[(b)] $\Q_{\lambda_0}$ is a hyperbola and $n(Q_{0}-Q_{b})+Q_{\lambda_0}-Q_a\sim0$;
		\item[(c)]  $\Q_{\lambda_0}$ is a hyperbola, $n$ is odd, and $n(Q_{0}-Q_{b})\sim0$.
	\end{itemize}
	Moreover, such trajectories are always symmetric with respect to the origin in Case (a) and, when $n$ is odd, in Case (b).
	They are symmetric with respect to the larger axis in Case (b) for even $n$, and with respect to the smaller axis in Case (c).
\end{theorem}
\begin{proof}
	%As Take $\{\alpha_1,\alpha_2\}=\{b,\lambda_0\}$, $\alpha_1<\alpha_2$.
	Let $M_0$ be the initial point of a given $n$-elliptic periodic trajectory, and $M_1$ the next point on the trajectory with the same elliptic coordinates.
	Then, integrating \eqref{eq:integral} from $M_0$ to $M_1$ along the trajectory, we get:
	$$
	n(Q_0-Q_{\alpha_1})+m(Q_{\alpha_2}-Q_{a})\sim0,
	$$
	where $m$ is the number of times that the particle crossed the $\mathsf{y}$-axis.
	Along that part of the trajectory, the particle crossed the $\mathsf{x}$-axis $m'$ times,
	where $m'=n$ if $b=\alpha_1$, and $m'=m$ if $b=\alpha_2$.
	
	To conclude the proof, we notice that at least one of $m$, $m'$ must be odd if the trajectory is not $n$-periodic.
\end{proof}

The explicit Cayley-type conditions for elliptic periodic trajectories are:
\begin{theorem}\label{th:elliptic-cayley}
	A billiard trajectory within $\E$ with the caustic $\Q_{\lambda_0}$ is $n$-elliptic periodic without being $n$-periodic if and only if one of the following conditions is satisfied:
	\begin{itemize}
		\item[(a)] $\Q_{\lambda_0}$ is an ellipse and
		\begin{gather*}
		C_1=0,
		\quad
		\left|
		\begin{array}{cc}
		C_1 & C_2
		\\
		C_2 & C_3
		\end{array}
		\right|=0,
		\quad
		\left|
		\begin{array}{ccc}
		C_1 & C_2 & C_3
		\\
		C_2 & C_3 & C_4
		\\
		C_3 & C_4 & C_5
		\end{array}
		\right|=0,
		\dots
		\quad\text{for}\quad n=2,4,6,\dots
		\\
		B_2=0,
		\quad
		\left|
		\begin{array}{cc}
		B_2 & B_3
		\\
		B_3 & B_4
		\end{array}
		\right|=0,
		\quad
		\left|
		\begin{array}{ccc}
		B_2 & B_3 & B_4
		\\
		B_3 & B_4 & B_5
		\\
		B_4 & B_5 & B_6
		\end{array}
		\right|=0,
		\dots
		\quad\text{for}\quad n=3,5,7,\dots;
		\end{gather*}
		\item[(b)] $\Q_{\lambda_0}$ is a hyperbola and
		\begin{gather*}
		D_1=0,
		\quad
		\left|
		\begin{array}{cc}
		D_1 & D_2
		\\
		D_2 & D_3
		\end{array}
		\right|=0,
		\quad
		\left|
		\begin{array}{ccc}
		D_1 & D_2 & D_3
		\\
		D_2 & D_3 & D_4
		\\
		D_3 & D_4 & D_5
		\end{array}
		\right|=0,
		\dots
		\quad\text{for}\quad n=2,4,6,\dots
		\\
		B_2=0,
		\quad
		\left|
		\begin{array}{cc}
		B_2 & B_3
		\\
		B_3 & B_4
		\end{array}
		\right|=0,
		\quad
		\left|
		\begin{array}{ccc}
		B_2 & B_3 & B_4
		\\
		B_3 & B_4 & B_5
		\\
		B_4 & B_5 & B_6
		\end{array}
		\right|=0,
		\dots
		\quad\text{for}\quad n=3,5,7,\dots;
		\end{gather*}	
		\item[(c)]  $\Q_{\lambda_0}$ is a hyperbola, $n$ is odd, and
		$$
		D_2=0,
		\quad
		\left|
		\begin{array}{cc}
		D_2 & D_3
		\\
		D_3 & D_4
		\end{array}
		\right|=0,
		\quad
		\left|
		\begin{array}{ccc}
		D_2 & D_3 & D_4
		\\
		D_3 & D_4 & D_5
		\\
		D_4 & D_5 & D_6
		\end{array}
		\right|=0,
		\dots
		\quad\text{for}\quad n=3,5,7,\dots.
		$$		
	\end{itemize}
	Here, we denoted:
	\begin{gather*}
	\frac{\sqrt{(a-x)(b-x)(\lambda_0-x)}}{b-x}=D_0+D_1x+D_2x^2+\dots,
	\end{gather*}
	the Taylor expansion around $x=0$, while $B$s and $C$s are as in Theorem \ref{th:cayley-billiard}.
\end{theorem}
\begin{proof}
	(a)
	In this case, the caustic $\Q_{\lambda_0}$ is an ellipse.
	
	Take first $n$ even.
	Using Theorem \ref{th:elliptic-periodic}, we have:
	$$
	nQ_0
	\sim
	nQ_{\lambda_0}-Q_b+Q_a\sim nQ_{\infty}-Q_b+Q_a
	\sim
	n Q_{\infty}-Q_{\infty}+Q_{\lambda_0}
	\sim
	(n-1)Q_{\infty}+Q_{\lambda_0}.
	$$
	The basis of $\mathcal{L}((n-1)Q_{\infty}+Q_{\lambda_0})$ is:
	$$
	1,x,x^2,\dots,x^{n/2-1},\frac{y}{x-\lambda_0},\frac{xy}{x-\lambda_0},
	\frac{x^{n/2-1}y}{x-\lambda_0},
	$$
	thus a non-trivial linear combination of these functions with a zero of order $n$ at $x=0$ exists if and only if:
	$$
	\left|
	\begin{array}{llll}
	C_{n/2} & C_{n/2-1} & \dots & C_1\\
	C_{n/2+1} & C_{n/2} & \dots & C_2\\
	\dots\\
	C_{n-1} & C_{n-2} & \dots & C_{n/2}
	\end{array}
	\right|
	=0.
	$$
	For odd $n$, we have:
	$$
	nQ_0
	\sim
	nQ_{\lambda_0}-Q_b+Q_a
	\sim
	(n-1)Q_{\infty}+Q_{\lambda_0}-Q_b+Q_a
	\sim
	(n-1) Q_{\infty}+Q_{\infty}
	\sim
	nQ_{\infty}.
	$$
	The basis of $\mathcal{L}(nQ_{\infty})$ is:
	$$
	1,x,x^2,\dots,x^{(n-1)/2},y,xy, x^{(n-1)/2-1}y,
	$$
	thus a non-trivial linear combination of these functions with a zero of order $n$ at $x=0$ exists if and only if:
	$$
	\left|
	\begin{array}{llll}
	B_{(n-1)/2+1} & B_{(n-1)/2} & \dots & B_2\\
	B_{(n-1)/2+2} & B_{(n-1)/2+1} & \dots & B_3\\
	\dots\\
	B_{n-1} & B_{n-2} & \dots & B_{(n-1)/2+1}
	\end{array}
	\right|
	=0.
	$$
	
	(b)
	In this case, the caustic $\Q_{\lambda_0}$ is a hyperbola.
	
	Take first $n$ even.
	Using Theorem \ref{th:elliptic-periodic}, we have:
	$$
	nQ_0
	\sim
	nQ_{b}-Q_{\lambda_0}+Q_a\sim nQ_{\infty}-Q_{\lambda_0}+Q_a
	\sim
	n Q_{\infty}-Q_{\infty}+Q_{b}
	\sim
	(n-1)Q_{\infty}+Q_{b}.
	$$
	The basis of $\mathcal{L}((n-1)Q_{\infty}+Q_{b})$ is:
	$$
	1,x,x^2,\dots,x^{n/2-1},\frac{y}{x-b},\frac{xy}{x-b},
	\frac{x^{n/2-1}y}{x-b},
	$$
	thus a non-trivial linear combination of these functions with a zero of order $n$ at $x=0$ exists if and only if:
	$$
	\left|
	\begin{array}{llll}
	D_{n/2} & D_{n/2-1} & \dots & D_1\\
	D_{n/2+1} & D_{n/2} & \dots & D_2\\
	\dots\\
	D_{n-1} & D_{n-2} & \dots & D_{n/2}
	\end{array}
	\right|
	=0.
	$$
	For odd $n$, we have:
	$$
	nQ_0
	\sim
	nQ_{b}+Q_{\lambda_0}-Q_a
	\sim
	(n-1)Q_{\infty}+Q_{b}+Q_{\lambda_0}-Q_a
	\sim
	(n-1) Q_{\infty}+Q_{\infty}
	\sim
	nQ_{\infty}.
	$$
	The determinant conditions are then obtained as in part (a).
	
	(c) We have $nQ_0\sim nQ_b\sim(n-1)Q_{\infty}+Q_b$.
	The conditions are derived as in the proof of Theorem \ref{th:cayley-billiard}, just replacing $\lambda_0$ by $b$.
\end{proof}

\begin{corollary}
	We notice that whenever a trajectory is $n$-elliptic periodic, the divisor $n(Q_0-Q_{\lambda_0})$ is of order $2$ on the elliptic curve \refeq{eq:billiard-cubic}, which means it is equivalent to one of the following: $Q_{\lambda_0}-Q_{\infty}$, $Q_{a}-Q_{\infty}$, or $Q_{b}-Q_{\infty}$.
	
	More precisely, a billiard trajectory  within $\E$ with the caustic $\Q_{\lambda_0}$ is $n$-elliptic periodic without being $n$-periodic if and only if:
	\begin{itemize}
		\item $\Q_{\lambda_0}$ is an ellipse and $n(Q_{0}-Q_{\lambda_0})\sim Q_{\lambda_0}-Q_{\infty}$;
		\item $\Q_{\lambda_0}$ is a hyperbola, $n$ even, and $n(Q_{0}-Q_{\lambda_0})\sim Q_{b}-Q_{\infty}$;
		\item  $\Q_{\lambda_0}$ is a hyperbola, $n$ odd, and $n(Q_{0}-Q_{\lambda_0})$ is equivalent to $ Q_{a}-Q_{\infty}$ or
		$Q_{\lambda_0}-Q_{\infty}$.
	\end{itemize}
\end{corollary}

\begin{remark}
	Theorem \ref{th:cayley-billiard} implies that a billiard trajectory  within $\E$ with the caustic $\Q_{\lambda_0}$ is $2n$-periodic if and only if $n(Q_0-Q_{\lambda_0})$ is of order $2$ or equivalent to the zero divisor.
	Thus, such a trajectory would be $2n$-periodic without being $n$-elliptic periodic if and only if:
	\begin{itemize}
		\item $\Q_{\lambda_0}$ is an ellipse and $n(Q_{0}-Q_{\lambda_0})$ is equivalent to $Q_{a}-Q_{\infty}$ or $ Q_{b}-Q_{\infty}$;
		\item $\Q_{\lambda_0}$ is a hyperbola, $n$ even, and $n(Q_{0}-Q_{\lambda_0})$ is equivalent to $Q_{a}-Q_{\infty}$ or $ Q_{\lambda_0}-Q_{\infty}$;
		\item  $\Q_{\lambda_0}$ is a hyperbola, $n$ odd, and $n(Q_{0}-Q_{\lambda_0})\sim Q_{b}-Q_{\infty}$.
	\end{itemize}
	Theorem \ref{th:even-periodic} shows that none of these scenarios can be realized.
\end{remark}

\section{Polynomial functional equations and periodicity conditions}\label{sec:poly}
\subsection{Pell's equations, extremal polynomials, and periodicity}

The matrix conditions for periodic trajectories, presented in Sections \ref{sec:cayley} and \ref{sec:elliptic-periodic}, can be equivalently written in a form of polynomial functional equations.
We will derive those equations in this section, and relate them to the so called \emph{Pell equations}.
For discussion including higher-dimensional cases, see \cite{DragRadn2018} and references therein.

\begin{theorem}\label{th:polynomial}
	The billiard trajectories within $\E$ with caustic $\C_{\lambda_0}$ are $n$-periodic if and only if there exists a pair of real polynomials $p_{d_1}$, $q_{d_2}$ of degrees $d_1$, $d_2$ respectively, and satisfying the following:
	\begin{itemize}
		\item[(a)] if $n=2m$ is even,  then $d_1=m$, $d_2=m-2$, and
		$$
		p_{m}^2(s)
		-
		s\left(s-\frac1a\right)\left(s-\frac1b\right)\left(s-\frac1{\lambda_0}\right)
		{q}_{m-2}^2(s)=1;
		$$
		\item[(b)] if $n=2m+1$ is odd, then $d_1=m$, $d_2=m-1$, and
		$$
		\left(s-\frac1{\lambda_0}\right)p_m^2(s)
		-
		s\left(s-\frac1a\right)\left(s-\frac1b\right)q_{m-1}^2(s)=-1.
		$$	
	\end{itemize}	
\end{theorem}
\begin{proof}
	We note first that the  proof of Theorem \ref{th:cayley-billiard} implies that there is a non-trivial linear combination of the bases \eqref{eq:basis-even} for $n$ even, or \eqref{eq:basis-odd} for $n$ odd,  with the zero of order $n$ at $x=0$.
	
	(a) For $n=2m$, from there we get that there are polynomials $p_m^*(x)$ and $q_{m-2}^*(x)$ of degrees $m$ and $m-2$ respectively, such that the expression
	$$
	p_{m}^*(x)-q_{m-2}^*(x)\sqrt{(a-x)(b-x)(\lambda_0-x)}
	$$
	has a zero of order $2m$ at $x=0$.
	Multiplying that expression by
	$$
	p_{m}^*(x)+q_{m-2}^*(x)\sqrt{(a-x)(b-x)(\lambda_0-x)},
	$$
	we get that the polynomial $(p_{m}^*(x))^2-(a-x)(b-x)(\lambda_0-x)(q_{m-2}^*(x))^{2}$ has a zero of order $2m$ at $x=0$.
	Since the degree of that polynomial is $2m$, is follows that:
	$$
	(p_{m}^*(x))^2-(a-x)(b-x)(\lambda_0-x)(q_{m-2}^*(x))^{2}=cx^{2m},
	$$
	for some constant $c$.
	Notice that $c$ is positive, since it equals the square of the leading coefficient of $p_m^*$.
	Dividing the last relation by $cx^{2m}$ and introducing $s=1/x$, we get the requested relation.
	
	(b) On the other hand, for $n=2m+1$, we get
	that there are polynomials $p_m^*(x)$ and $q_{m-1}^*(x)$ of degrees $m$ and $m-1$ respectively, such that the expression
	$$
	p_{m}^*(x)-q_{m-1}^*(x)\frac{\sqrt{(a-x)(b-x)(\lambda_0-x)}}{\lambda_0-x}
	$$
	has a zero of order $2m+1$ at $x=0$.
	Multiplying that expression by
	$$
	(\lambda_0-x)
	\left(
	p_{m}^*(x)+q_{m-1}^*(x)\frac{\sqrt{(a-x)(b-x)(\lambda_0-x)}}{\lambda_0-x}
	\right)
	,
	$$
	we get that the polynomial $(\lambda_0-x)(p_{m}^*(x))^2-(a-x)(b-x)(q_{m-1}^*(x))^{2}$ has a zero of order $2m+1$ at $x=0$.
	Since the degree of that polynomial is $2m+1$, is follows that:
	$$
	(\lambda_0-x)(p_{m}^*(x))^2-(a-x)(b-x)(q_{m-1}^*(x))^{2}=cx^{2m+1},
	$$
	for some constant $c$.
	Notice that $c$ is positive, since it equals the opposite of the square of the leading coefficient of $p_m^*$.
	Dividing the last relation by $-cx^{2m+1}$ and introducing $s=1/x$, we get the requested relation.
\end{proof}	

\begin{corollary}\label{cor:pell-periodic}
	If the billiard trajectories within $\E$ with caustic $\C_{\lambda_0}$ are $n$-periodic, then there exist real polynomials $\hat{p}_n$ and $\hat{q}_{n-2}$ of degrees $n$ and $n-2$ respectively, which satisfy the Pell  equation:
	\begin{equation}\label{eq:pell}
	\hat{p}_{n}^2(s)-s\left(s-\frac1a\right)\left(s-\frac1b\right)\left(s-\frac1{\lambda_0}\right)\hat{q}_{n-2}^2(s)=1.
	\end{equation}
\end{corollary}
\begin{proof}
	For $n=2m$, take $\hat{p}_n=2p_{m}^2-1$ and $\hat{q}_{n-2}=2p_mq_{m-2}$.
	For $n=2m+1$, we set $\hat{p}_n=2\left(s-\dfrac1{\lambda_0}\right)p_{m}^2+1$ and $\hat{q}_{n-2}=2p_mq_{m-1}$.
\end{proof}

\begin{theorem}\label{th:polynomial-elliptic}
	The billiard trajectories within $\E$ with caustic $\C_{\lambda_0}$ are elliptic $n$-periodic without being $n$-periodic if and only if there exists a pair of real polynomials $p_{d_1}$, $q_{d_2}$ of degrees $d_1$, $d_2$ respectively, and satisfying the following:
	\begin{itemize}
		\item[(a)] $\C_{\lambda_0}$ is an ellipse and
		\begin{itemize}
			\item $n=2m$ is even, $d_1=d_2=m-1$,
			$$
			s\left(s-\frac1{\lambda_0}\right)p_{m-1}^2(s)
			-\left(s-\frac1a\right)\left(s-\frac1b\right)q_{m-1}^2(s)=-1;
			$$
			\item $n=2m+1$ is odd, $ d_1=m$, $d_2=m-1$,
			$$
			s p_{m}^2(s)
			-\left(s-\frac1a\right)\left(s-\frac1b\right)\left(s-\frac1{\lambda_0}\right)q_{m-1}^2(s)=1;
			$$
		\end{itemize}	
		\item[(b)] $\C_{\lambda_0}$ is a hyperbola and
		\begin{itemize}
			\item $n=2m$ is even, $d_1=d_2=m-1$,
			$$
			s\left(s-\frac1{b}\right)p_{m-1}^2(s)
			-\left(s-\frac1a\right)\left(s-\frac1{\lambda_0}\right)q_{m-1}^2(s)=-1;
			$$
			\item $n=2m+1$ is odd, $ d_1=m$, $d_2=m-1$,
			$$
			s p_{m}^2(s)
			-\left(s-\frac1a\right)\left(s-\frac1b\right)\left(s-\frac1{\lambda_0}\right)q_{m-1}^2(s)=1;
			$$
		\end{itemize}
		\item[(c)] $\C_{\lambda_0}$ is a hyperbola, $n=2m+1$ is odd, $d_1=m$, $d_2=m-1$,
		$$
		\left(s-\frac1b\right) p_{m}^2(s)
		-s\left(s-\frac1a\right)\left(s-\frac1{\lambda_0}\right)q_{m-1}^2(s)=-1.
		$$
	\end{itemize}	
\end{theorem}
\begin{proof}
	(a) For $n=2m$, the proof of Theorem \ref{th:elliptic-cayley} implies that there are polynomials $p_{m-1}^*(x)$ and $q_{m-1}^*(x)$ of degrees $m-1$, such that the expression
	$$
	p_{m-1}^*(x)-q_{m-1}^*(x)\frac{\sqrt{(a-x)(b-x)(\lambda_0-x)}}{\lambda_0-x}
	$$
	has a zero of order $2m$ at $x=0$.
	Multiplying that expression by
	$$
	(\lambda_0-x)\left(p_{m-1}^*(x)+q_{m-1}^*(x)\frac{\sqrt{(a-x)(b-x)(\lambda_0-x)}}{\lambda_0-x}\right),
	$$
	we get that the polynomial
	$(\lambda_0-x)(p_{m-1}^*(x))^2-(a-x)(b-x)(q_{m-1}^*(x))^2$ has a zero of order $2m$ at $x=0$.
	Since the degree of that polynomial is $2m$, is follows that:
	$$
	(\lambda_0-x)(p_{m-1}^*(x))^2-(a-x)(b-x)(q_{m-1}^*(x))^2=cx^{2m},
	$$
	for some constant $c$.
	Notice that $c$ is negative, since it is opposite to the square of the leading coefficient of $q_{m-1}^*$.
	Dividing the last relation by $-cx^{2m}$ and introducing $s=1/x$, we get the requested relation.
	
	For $n=2m+1$, the proof of Theorem \ref{th:elliptic-cayley} implies that there are polynomials $p_{m}^*(x)$ and $q_{m-1}^*(x)$ of degrees $m$ and $m-1$, such that the expression
	$$
	p_{m}^*(x)-q_{m-1}^*(x)\sqrt{(a-x)(b-x)(\lambda_0-x)}
	$$
	has a zero of order $2m+1$ at $x=0$.
	Multiplying that expression by
	$$
	p_{m}^*(x)+q_{m-1}^*(x)\sqrt{(a-x)(b-x)(\lambda_0-x)}
	,
	$$
	we get that the polynomial
	$(p_{m}^*(x))^2-(a-x)(b-x)(\lambda_0-x)(q_{m-1}^*(x))^2$ has a zero of order $2m+1$ at $x=0$.
	Since the degree of that polynomial is $2m+1$, is follows that:
	$$
	(p_{m}^*(x))^2-(a-x)(b-x)(\lambda_0-x)(q_{m-1}^*(x))^2=cx^{2m+1}
	$$
	for some constant $c$.
	Notice that $c$ is positive, since it equals the square of the leading coefficient of $q_{m-1}^*$.
	Dividing the last relation by $cx^{2m+1}$ and introducing $s=1/x$, we get the requested relation.
	
	The result (b) is obtained in a similar way.
	
	For (c), the proof of Theorem \ref{th:elliptic-cayley} implies that there are polynomials $p_{m}^*(x)$ and $q_{m-1}^*(x)$ of degrees $m$ and $m-1$, such that the expression
	$$
	p_{m}^*(x)-q_{m-1}^*(x)\frac{\sqrt{(a-x)(b-x)(\lambda_0-x)}}{b-x}
	$$
	has a zero of order $2m+1$ at $x=0$.
	Multiplying that expression by
	$$
	(b-x)\left(p_{m}^*(x)+q_{m-1}^*(x)\frac{\sqrt{(a-x)(b-x)(\lambda_0-x)}}{b-x}\right)
	,
	$$
	we get that the polynomial
	$(b-x)(p_{m}^*(x))^2-(a-x)(\lambda_0-x)(q_{m-1}^*(x))^2$ has a zero of order $2m+1$ at $x=0$.
	Since the degree of that polynomial is $2m+1$, is follows that:
	$$
	(b-x)(p_{m}^*(x))^2-(a-x)(\lambda_0-x)(q_{m-1}^*(x))^2=cx^{2m+1}
	$$
	for some constant $c$.
	Notice that $c$ is negative, since it is opposite to the square of the leading coefficient of $p_{m}^*$.
	Dividing the last relation by $-cx^{2m+1}$ and introducing $s=1/x$, we get the requested relation.
	
\end{proof}

From Corollary \ref{cor:pell-periodic} and Theorem \ref{th:polynomial-elliptic} we get:

\begin{corollary}\label{cor:pell-elliptic}
	The billiard trajectories within $\E$ with caustic $\C_{\lambda_0}$ are $n$-elliptic periodic if and only if there exist a pair of real polynomials $\hat{p}_n$, $\hat{q}_{n-2}$ of degrees $n$ and $n-2$ respectively, such that the Pell's equation \eqref{eq:pell} holds.
\end{corollary}
\begin{proof}
	If the trajectories are $n$-periodic, the statement is true according to Corollary \ref{cor:pell-periodic}.
	
	If the trajectories are $n$-elliptic periodic without being $n$-periodic,
	then one of the cases of Theorem \ref{th:polynomial-elliptic} is satisfied.
	
	In the case (a), $n=2m$, we have $\hat{p}_n=2s\left(s-\dfrac1{\lambda_0}\right)p_{m-1}^2(s)+1$,
	$\hat{q}_{n-2}(s)=2p_{m-1}(s)q_{m-1}(s)$.
	For $n=2m+1$, we have $\hat{p}_n=2sp_{m}^2(s)-1$,
	$\hat{q}_{n-2}(s)=2p_{m}(s)q_{m-1}(s)$.
	
	In the case (b), $n=2m$, we have $\hat{p}_n=2s\left(s-\dfrac1b\right)p_{m-1}^2(s)+1$,
	$\hat{q}_{n-2}(s)=2p_{m-1}(s)q_{m-1}(s)$.
	For $n=2m+1$, it is the same as in (a).
	
	In the case (c), we have $\hat{p}_n=2\left(s-\dfrac1b\right)p_{m}^2(s)+1$,
	$\hat{q}_{n-2}(s)=2p_{m}(s)q_{m-1}(s)$.

	Now, suppose that the Pell equation is satisfied.
	Then, according to part (a) of Theorem \ref{th:polynomial}, the trajectories are $2n$-periodic, and by Theorem \ref{th:even-periodic} they are $n$-elliptic periodic.
\end{proof}

The next two theorems address the question how to distinguish $n$-elliptic periodic trajectories which are not $n$-periodic from those that are.

\begin{theorem}\label{th:odd-pell}
	Suppose that real polynomials $\hat{p}_{2n+1}$, $\hat{q}_{2n-1}$ of degrees $2n+1$, $2n-1$ respectively satisfy the Pell equation:
	\begin{equation}\label{eq:pell-odd}
	\hat{p}_{2n+1}^2(s)
	-
	s\left(s-\frac1a\right)\left(s-\frac1b\right)\left(s-\frac1{\lambda_0}\right)
	\hat{q}_{2n-1}^2(s)=1.
	\end{equation}
	Then there exist unique real polynomials $p_n$, $q_{n-1}$, $S_1$, $S_3$ of degrees $n$, $n-1$, $1$, $3$ respectively, such that:
	\begin{itemize}
		\item
		$S_1$, $S_3$ are monic and
		$
		S_1S_3=s\left(s-\dfrac1a\right)\left(s-\dfrac1b\right)\left(s-\dfrac1{\lambda_0}\right);
		$
		
		\item $S_1p_n^2-S_3q_{n-1}^2$ equals $1$ or $-1$.
	\end{itemize}
	Moreover, the billiard trajectories within $\E$ with caustic $\C_{\lambda_0}$ are all $(2n+1)$-elliptic periodic. They are $(2n+1)$-periodic if and only if  $S_1(s)=s-\dfrac1{\lambda_0}$
\end{theorem}

\begin{proof}
	Equation \eqref{eq:pell-odd} is equivalent to:
	$$
	(\hat{p}_{2n+1}(s)-1)(\hat{p}_{2n+1}(s)+1)=s\left(s-\frac1a\right)\left(s-\frac1b\right)\left(s-\frac1{\lambda_0}\right)\hat{q}_{2n-1}^2(s).
	$$
	Notice that the two factors, $\hat{p}_{2n+1}(s)-1$ and $\hat{p}_{2n+1}(s)+1$, of the lefthand side are mutually prime, which implies that polynomial $\hat{p}_{2n+1}$ takes one of the values $1$ and $-1$ at three of the points from the set $\{0,1/a,1/b,1/\lambda_0\}$, and the opposite value at the remaining point of that set.
	Replacing $\hat{p}_{2n+1}$ by $-\hat{p}_{2n+1}$ if needed, we can assume that $\hat{p}_{2n+1}(s_1)=1$ at only one point $s_1\in\{0,1/a,1/b,1/\lambda_0\}$.
	
	Now we set $S_1$ and $S_3$ as:
	$$
	S_1(s)=s-s_1,
	\quad
	S_1(s)S_3(s)=s\left(s-\dfrac1a\right)\left(s-\dfrac1b\right)\left(s-\dfrac1{\lambda_0}\right).
	$$
	Polynomials $p_n$, $q_{n-1}$ are such that:
	$$
	\hat{p}_{2n+1}-1=\sigma \cdot 2 S_1 p_n^2,
	\quad
	\hat{p}_{2n+1}+1=\sigma \cdot 2 S_3 q_{n-1}^2,
	\quad
	\hat{q}_{2n-1}=2p_nq_{n-1},
	\quad
	\sigma\in\{-1,1\}.
	$$
	From there, we get:
	\begin{equation}\label{eq:sigma}
	S_1p_n^2-S_3q_n^2=-\sigma.
	\end{equation}
	
	Denote $\{c_1,c_2,c_3,c_4\}=\{0,1/a,1/b,1/\lambda_0\}$, so that $c_4<c_3<c_2<c_1$.
	Since $0<b<a$ and $0<\lambda_0<a$, we have $c_4=0$, and $c_1\in\{1/b,1/\lambda_0\}$.
	If $c_2$ or $c_3$ is the root of $S_1$, then $S_1$ is negative on $[c_4,c_3]$ and positive on $[c_2,c_1]$, while $S_3$ has the opposite signs on these two segments.
	Thus $S_1p_n^2-S_3q_{n-1}^2$ is negative on $[c_4,c_3]$, and positive on $[c_2,c_1]$, which is not possible.
	We conclude that the root of $S_1$ can be only $c_1$ or $c_4$.
	
	If $S_1(s)=s-c_4=s$, then $-\sigma$ equals the free coefficient of $-S_3q_n^2$, which is positive, thus $-\sigma=1$, and \eqref{eq:sigma} becomes:
	$$
	sp_n^2-\left(s-\dfrac1a\right)\left(s-\dfrac1b\right)\left(s-\dfrac1{\lambda_0}\right)q_n^2=1.
	$$
	Now applying cases (a) and (b) of Theorem \ref{th:polynomial-elliptic}, we get that the billiard trajectories within $\E$ with caustic $\C_{\lambda_0}$ are $(2n+1)$-elliptic periodic without being $(2n+1)$-periodic.

	If $S_1(s)=s-c_1$, then $-\sigma$ equals the free coefficient of $S_1p_n^2$, which is negative, thus $-\sigma=-1$.
	Thus, when $c_1=1/b$, i.e.~when $\C_{\lambda_0}$ is a hyperbola, \eqref{eq:sigma} becomes:
	$$
	\left(s-\frac1b\right)p_n^2-s\left(s-\dfrac1a\right)\left(s-\dfrac1{\lambda_0}\right)q_n^2=1,
	$$
	so the billiard trajectories within $\E$ with caustic $\C_{\lambda_0}$ are elliptic $n$-periodic without being $n$-periodic, according to case (c) of Theorem \ref{th:polynomial-elliptic}.
	
	If $c_1=1/\lambda_0$, which means that $\C_{\lambda_0}$ is an ellipse, \eqref{eq:sigma} becomes:
	$$
	\left(s-\frac1{\lambda_0}\right)p_n^2-s\left(s-\dfrac1a\right)\left(s-\dfrac1b\right)q_n^2=1,
	$$
	so the billiard trajectories within $\E$ with caustic $\C_{\lambda_0}$ are  $(2n+1)$-periodic, according to case (b) of Theorem \ref{th:polynomial}.
\end{proof}

\begin{remark}
	It is interesting here to notice, that the proof of Theorem \ref{th:odd-pell} represents an alternative, algebraic, proof of Lemma \ref{lemma:hyperbola}.
\end{remark}

\begin{theorem}\label{th:even-pell}
	Suppose that real polynomials $\hat{p}_{2n}$, $\hat{q}_{2n-2}$ of degrees $2n$, $2n-2$ respectively satisfy the Pell's equation:
	\begin{equation}\label{eq:pell-even}
	\hat{p}_{2n}^2(s)
	-
	s\left(s-\frac1a\right)\left(s-\frac1b\right)\left(s-\frac1{\lambda_0}\right)
	\hat{q}_{2n-2}^2(s)=1.
	\end{equation}
	Then there exist unique real polynomials $p_{d_1}$, $q_{d_2}$, $S'$, $S''$ of degrees $d_1$, $d_2$, $d'$, $d''$ respectively, such that either $d_1=d_2=n-1$, $d'=d''=2$ or $d_1=n$, $d_2=n-2$, $d'=0$, $d''=4$, and:
	\begin{itemize}
		\item
		$S'$, $S''$ are monic and
		$
		S'S''=s\left(s-\dfrac1a\right)\left(s-\dfrac1b\right)\left(s-\dfrac1{\lambda_0}\right);
		$
		\item $S'p_{d_1}^2-S''q_{d_2}^2$ equals $1$ or $-1$.
	\end{itemize}
	Moreover, the billiard trajectories within $\E$ with caustic $\C_{\lambda_0}$ are all $2n$-elliptic periodic. They are $2n$-periodic if and only if  $S'=1$.
\end{theorem}

\begin{proof}
	Equation \eqref{eq:pell-even} is equivalent to:
	$$
	(\hat{p}_{2n}(s)-1)(\hat{p}_{2n}(s)+1)=s\left(s-\frac1a\right)\left(s-\frac1b\right)\left(s-\frac1{\lambda_0}\right)\hat{q}_{2n-2}^2(s).
	$$
	Notice that the two factors, $\hat{p}_{2n}(s)-1$ and $\hat{p}_{2n}(s)+1$, of the lefthand side are mutually prime, which implies that polynomial $\hat{p}_{2n}$ takes the value $1$ at even number of points from the set $\{0,1/a,1/b,1/\lambda_0\}$,  and the opposite value $-1$ at the remaining points of that set.
	
	First, suppose that $\hat{p}_{2n}(s)$ takes the same value at each point of those for points.
	Replacing $\hat{p}_{2n}$ by $-\hat{p}_{2n}$ if needed, we can assume that value is $-1$.
	Set $S'$ and $S''$ as:
	$$
	S'(s)=1,
	\quad
	S''(s)=s\left(s-\dfrac1a\right)\left(s-\dfrac1b\right)\left(s-\dfrac1{\lambda_0}\right).
	$$
	Polynomials $p_n$, $q_{n-2}$ are such that:
	$$
	\hat{p}_{2n}-1=\sigma\cdot 2 p_n^2,
	\quad
	\hat{p}_{2n}+1=\sigma\cdot 2 S'' q_{n-2}^2,
	\quad
	\hat{q}_{2n-2}=2p_nq_{n-2},
	\quad
	\sigma\in\{-1,1\},
	$$
	from where we get
	$p_n^2-S''q_{n-2}^2=-\sigma$.
	We can conclude that $-\sigma=1$, since it equals the square of the free coefficient of $p_n$.	Finally, applying case (a) of Theorem \ref{th:polynomial}, we can conclude that the billiard trajectories within $\E$ with caustic $\C_{\lambda_0}$ are $2n$-periodic.

	Second, we will suppose that $\hat{p}_{2n}(s)$  takes each of the values $-1$, $1$ at two points from the set $\{0,1/a,1/b,1/\lambda_0\}$.
	Replacing $\hat{p}_{2n}$ by $-\hat{p}_{2n}$ if needed, we can assume that $\hat{p}_{2n}(0)=\hat{p}_{2n}(s_1)=1$, for $s_1\in\{1/a,1/b,1/\lambda_0\}$.
	Set $S'$, $S''$ as:
	$$
	S'(s)=s(s-s_1),
	\quad
	S'(s)S''(s)=s\left(s-\dfrac1a\right)\left(s-\dfrac1b\right)\left(s-\dfrac1{\lambda_0}\right).
	$$
	Polynomials $p_{n-1}$, $q_{n-1}$ are such that:
	$$
	\hat{p}_{2n}-1=\sigma\cdot2S'p_{n-1}^2,
	\quad
	\hat{p}_{2n}+1=\sigma\cdot2S''q_{n-1}^2,
	\quad
	\hat{q}_{2n-2}=2p_{n-1}q_{n-1},
	\quad
	\sigma\in\{-1,1\},
	$$
	so $S'p_{n-1}^2-S''q_{n-1}^2=-\sigma$.
	From there, $-\sigma$ equals the free coefficient of $-S''q_{n-1}^2$, which is negative, thus $-\sigma=-1$:
	\begin{equation}\label{eq:sigma2}
	S' p_n^2-S''q_{n-2}^2=-1.
	\end{equation}
	As in the proof of Theorem \ref{th:odd-pell}, we denote by $c_4<c_3<c_2<c_1$ the elements of $\{0,1/a,1/b,1/\lambda_0\}$.
	If $c_2$ or $c_3$ is the root of $S'$, then $S'$ is negative on $[c_4,c_3]$ and positive on $[c_2,c_1]$, while $S''$ has the opposite signs on these two segments.
	Thus $S'p_{n-1}^2-S''q_{n-1}^2$ is negative on $[c_4,c_3]$, and positive on $[c_2,c_1]$, which is not possible.
	We conclude that $s_1=c_1$, which can be either $1/\lambda_0$ or $1/b$.
	
	If $c_1=1/\lambda_0$, that means $\C_{\lambda_0}$ is an ellipse, while it is a hyperbola for $c_1=1/b$.
	Thus, \eqref{eq:sigma2} implies that
	cases (a) and (b) respectively of Theorem \ref{th:polynomial} are satisfied, i.e.~the billiard trajectories within $\E$ with caustic $\C_{\lambda_0}$ are $2n$-elliptic periodic without being $2n$-periodic.
\end{proof}

\subsection{Generalized Chebyshev polynomials and rotation function}

From previous consideration we see that a caustic $\C_{\lambda_0}$ generates $n$-elliptic periodic trajectories within $\E$ if and only if there exist a pair of real polynomials
$\hat{p}_n$, $\hat{q}_{n-2}$
of degrees $n$ and $n-2$ respectively such that the Pell equation \eqref{eq:pell} holds:
$$
\hat{p}_n^2(s)-\hat{\mathcal{P}}_{4}(s)\hat{q}_{n-2}^2(s)=1.
$$
Here $\hat{\mathcal{P}}_{4}(s)=\prod_{i=1}^4(s-c_i)$, assuming that
$\{c_1, c_2, c_3, c_4\}=\{1/\lambda_0, 1/a, 1/b, 0\}$ are, as in the proofs of Theorems \ref{th:odd-pell} and \ref{th:even-pell}, ordered by the condition
$
c_4=0<c_3<c_2<c_1.
$

The polynomials $\hat {p}_n$ are so called \emph{generalized Chebyshev polynomials} on two intervals $[c_4, c_3]\cup [c_2, c_1]$, with an appropriate normalization. Namely, one can consider the question of finding the monic polynomial of certain degree  $n$ which minimizes the maximum norm on the union of two intervals. Denote such a polynomial as $\hat P_n$ and its norm $L_n$. The fact that polynomial $\hat {p}_n$ is a solution of the Pell equation on the union of intervals $[c_4, c_3]\cup [c_2, c_1]$ is equivalent to the
following conditions:
\begin{itemize}
	\item[(i)] $\hat {p}_n=\hat {P}_n/\pm L_n$
	\item[(ii)] the set $[c_4, c_3]\cup [c_2, c_1]$ is the maximal subset of $\mathbf R$ for which $\hat {P}_n$
	is the minimal polynomial in the sense above.
\end{itemize}

Chebyshev was the first who considered a similar problem on one interval, and this was how celebrated Chebyshev polynomials emerged in XIXth century. We are going to say a bit more about original Chebyshev polynomials
in Section \ref{sec:classicalextrema}.
Now, following the principles formulated by Chebyshev and his school and also Borel (see \cite{AhiezerAPPROX}), we are going to study the structure of extremal points of $\hat{p}_n$, in particular the set of points of alternance.

Notice that the roots of $\hat{\mathcal{P}}_{4}(s)$ are simple solutions of the equation $\hat{p}_n^2(s)=1$, while the roots of
$\hat{q}_{n-2}(s)$ are double solutions of the equation $\hat{p}_n^2(s)=1$.
Because of the degrees of the polynomials, these are all points where $\hat{p}_n^2(s)$ equals to unity.

Let us recall that a set of \emph{points of alternance} is, by definition, a subset of the solutions of the equation $\hat{p}_n^2(s)=1$, with the maximal number of elements, such that the signs of $\hat{p}_n$ alter on it.
Such a set is not uniquely determined, however the number of its elements if fixed and equal to $n+1$.

If we denote the number of points of alternance of the polynomial $\hat{p}_n$ on the segment $[c_{4},c_{3}]$ as  $1+m_1$ and on the segment $[c_{4},c_{1}]$ as  $1+m_0$, we see that
the difference $m_{0}-m_1$ is thus equal to the number of points of alternance on the interval $[c_{3},c_{1}]$.
According to the structure of the sets of the alternance, that number equals the sum of the numbers of the double points of alternance from the interval $(c_{2},c_{1})$ and one simple point of alternance at one of the endpoints of the interval. Thus we get
$$
m_{0}=m_{1}+\tau_1 +1,
$$
and
$$
m_0> m_1.
$$
Here $\tau_1$ is the number of zeros of the polynomial  $\hat {q}_{n-2}$ on the interval $(c_{2},c_{1})$.
The pair $(\tau_1, \tau_2)$, with $\tau_2 = m_1-1$  is called \emph{the signature}, see \cite{RR2014}.

From \cite{KLN1990} it follows that the numbers $(n=m_0, m_1)$ satisfy the following condition:
\begin{equation}\label{eq:windingKLN}
n\int_{c_1}^{\infty}\frac{1}{\sqrt{\hat{\mathcal{P}}_{4}(s)}}ds
=
m_1\int_{c_{3}}^{c_{2}}\frac{1}{\sqrt{\hat{\mathcal{P}}_{4}(s)}}ds.
\end{equation}
Thus $(m_0, m_1)$ will represent exactly the winding numbers, which we introduced at the end of Section \ref{sec:confocal}, of the corresponding billiard trajectories.

For the reader's sake, we are going to review briefly the results about the winding numbers from \cite{DragRadn2018} about winding numbers, specialized for the planar case.

\begin{lemma}[Theorem 2.12 from \cite{PS1999}]\label{lemma:pell-unique}
	Let $p_n$, $p_n^*$ be two polynomials of degree $n$, which solve the Pell's equations.
	Denote by
	$$
	\I=[c_4,c_3]\cup[c_2,c_1]
	\quad\text{and}\quad
	\I^*=[c_4^*,c_3^*]\cup[c_2^*,c_1^*]
	$$
	respectively the sets $\{x\mid |p_n(x)|\le 1\}$ and $\{x\mid |p_n^*(x)|\le 1\}$.
	Suppose that:
	\begin{itemize} \item [\textrm{(i)}] at least one of the segments from $\I$ coincides with a segment from $\I^*$;
		\item[\textrm{(ii)}] the other pair of segments from $\I$ and $\I^*$ have joint either left or right endpoint;
		\item[\textrm{(iii)}] in each pair of the corresponding segments $[c_4,c_3]$, $[c_4^*,c_3^*]$ and $[c_2,c_1]$, $[c_2^*,c_1^*]$,
		the polynomials  $p_n$, $p_n^*$  have the same number of extreme points.
	\end{itemize}
	Then the polynomials $p_n$, $p_n^*$  coincide up to a constant multiplier and sets $\I$ and $\I^*$ coincide.
\end{lemma}

\begin{theorem}\label{th:signature-caustics}
	Let $m_0$, $m_1$ be given integers.
	Then there is at most one ellipse $\E'$ and at most one hyperbola $\pazocal{H}$ from the confocal family \eqref{eq:confocal}, such that the billiard trajectories within $\E$ and with caustics $\E'$ and $\pazocal{H}$ are periodic with winding numbers $(m_0,m_1)$.
\end{theorem}

\begin{proof}
	All assumptions of Lemma \ref{lemma:pell-unique} are satisfied.
\end{proof}

\begin{corollary}\label{cor:rotation2}
	The rotation number
	$$
	\rho(\lambda)
	=
	\rho(\lambda, a, b)
	=
	\frac
	{\int_0^{\min\{b, \lambda\}}\frac{dt}{\sqrt{(\lambda-t)(b-t)(a-t)}}}
	{\int^0_{\max\{b, \lambda\}}\frac{dt}{\sqrt{(\lambda-t)(b-t)(a-t)}}}
	$$
	is a  strictly monotone function on each of the intervals $(-\infty, b)$ and $(b, a)$.
\end{corollary}

\begin{proof}
	First, observe that the rotation number is rational if and only if the billiard trajectories within $\E$ with caustic $\C_{\lambda}$ are periodic, when $\rho(\lambda)=m_1/m_0$, see equation (\ref{eq:windingKLN}).
	According to Theorem \ref{th:signature-caustics}, $\rho$ can take any rational value at most once on each of the intervals $(-\infty, b)$ and $(b, a)$.
	In addition, the rotation  map for the billiard within ellipse is a diffeomorphism at an open dense subset of each of the intervals $(-\infty, b)$ and $(b, a)$, which can be proved by using the same arguments as in \cite{PT2011}.
	From there, $\rho$ will be  one-to-one on the whole intervals, thus, since it is continuous, also strictly monotone.
\end{proof}

\begin{remark}
	In \cite{DuistermaatBOOK}, there is another proof of Corollary \ref{cor:rotation2}, which uses the theory of algebraic surfaces.
	A different proof of injectivity of rotation map, which works for higher-dimensional situations
	as well was presented recently in \cite{DragRadn2018}.
\end{remark}

\section{Trajectories with small periods: $n=3, 4, 5, 6$}\label{sec:examples}

\subsection{$3$-periodic trajectories}

There is a $3$-periodic trajectory of the billiard within $\E$, with a non-degenerate caustic $\C_{\lambda_0}$ if and only if the following conditions are satisfied:
\begin{itemize}
	\item the caustics is an ellipse, i.e.~$\lambda_0\in(0,b)$; and
	\item $C_2=0$.
\end{itemize}
We can calculate:
\begin{equation}\label{eq:C2}
C_2=\frac{(a-b)^2\lambda_0 ^2 + 2ab(a+b)\lambda_0-3 a^2 b^2}{8(  a b)^{3/2}\lambda_0^{5/2}} ,
\end{equation}
so $C_2=0$ is equivalent to:
$$
\lambda_0=-\frac{ab(a+b)\pm2ab\sqrt{a^2-ab+b^2}}{(a-b)^2}.
$$
Both solutions are always real, one of them being negative, and the other positive and smaller than $b$.
Thus there is a unique caustic giving $3$-peroidic trajectories that corresponds to real motion, the ellipse defined with:
\begin{equation}\label{eq:lambda3}
\lambda_0=-\frac{ab(a+b)+2ab\sqrt{a^2-ab+b^2}}{(a-b)^2}.
\end{equation}
The winding numbers of such trajectories satisfy $m_0>m_1$, with $m_0=3$ and $m_1$ being even.
Thus, $(m_0,m_1)=(3,2)$, $(\tau_1,\tau_2)=(0,1)$.
The graph of the corresponding polynomial $\hat{p}_3(s)$ is shown in Figure \ref{fig:p3}.
\begin{figure}[h]
	\begin{center}
\begin{tikzpicture}

\draw[gray](0,1)--(5,1);
\draw[gray](0,-1)--(5,-1);

\draw[gray](0.95,0)--(0.95,1);
\draw[gray](5,0)--(5,1);
\draw[gray](2.25,0)--(2.25,-1);
\draw[gray](3.88,0)--(3.88,-1);

\draw [very thick] plot [smooth, tension=1] coordinates { (0,-1) (1,0.99) (3,-1.5) (5,1)};

%\draw[very thick, red] (0,-1) to[out=80,in=180] (1,1);
%\draw[very thick, red] (1,1) to[out=0,in=180] (3,-1.5);

\draw[thick] (-0.5,0)--(5.5,0);
\draw[thick] (0,-2)--(0,2);

\fill[black] (0, 1) circle (2pt);
\draw (-0.2,1) node {$1$};

\fill[black] (0, -1) circle (2pt);
\draw (-0.4,-1) node {$-1$};

\fill[black] (0.95, 1) circle (2pt);
\fill[black] (0.95, 0) circle (2pt);
\draw (0.95,-0.35) node {$\gamma$};

\fill[black] (5, 1) circle (2pt);
\fill[black] (5, 0) circle (2pt);
\draw (5,-0.35) node {$c_1$};

\fill[black] (2.25, -1) circle (2pt);
\fill[black] (2.25, 0) circle (2pt);
\draw (2.25,0.35) node {$c_3$};

\fill[black] (3.88, -1) circle (2pt);
\fill[black] (3.88, 0) circle (2pt);
\draw (3.88,0.35) node {$c_2$};

\end{tikzpicture}

		\caption{The graph of $\hat{p}_3(s)$. The parameters are $c_1=1/\lambda_0$, $c_2=1/b$, $c_3=1/a$.}\label{fig:p3}
	\end{center}
\end{figure}
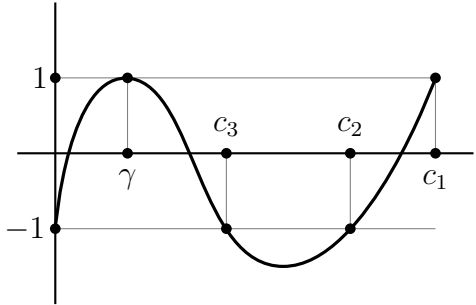

\subsection{$4$-periodic trajectories}

There is a $4$-periodic trajectory of the billiard within $\E$, with non-degenerate caustic $\C_{\lambda_0}$ if and only if $B_3=0$.

We can calculate:
\begin{equation}\label{eq:B3}
B_3=
\frac{(a \lambda_0 -a b+\lambda_0  b) (a \lambda_0 +a b-\lambda_0  b) (-a \lambda_0 +a b+\lambda_0  b)}{16(ab\lambda_0)^{5/2}}
,
\end{equation}
so the coefficient equals zero if and only if:
\begin{equation}\label{eq:c4lambda}
\lambda_0
\in
\left\{
\frac{ab}{a+b},\frac{ab}{a-b},\frac{ab}{b-a}
\right\}
.
\end{equation}
The first solution $\lambda_0=ab/(a+b)$ is positive and smaller than $b$, so it corresponds to a confocal ellipse as a caustic.
The second solution $\lambda_0=ab/(a-b)$ is always bigger than $b$.
It is also smaller than $a$ if and only if $b<a/2$, so this is when a confocal hyperbola as a caustic for a $4$-periodic trajectory exists.
The third solution is negative, so it does not correspond to any real trajectories.

The winding numbers of $4$-periodic trajectories satisfy $m_0>m_1$, with $m_0=4$ and $m_1$ being even.
Thus, $(m_0,m_1)=(4,2)$, $(\tau_1,\tau_2)=(1,1)$.
The graph of the corresponding polynomial $\hat{p}_4(s)$ is shown in Figure \ref{fig:p4}.
\begin{figure}[h]
	\begin{center}
\begin{tikzpicture}

\draw[gray](0,1)--(4.9,1);
\draw[gray](0,-1)--(6,-1);

\draw[gray](0.95,0)--(0.95,1);
\draw[gray](4.9,0)--(4.9,1);
\draw[gray](2.26,0)--(2.26,-1);
\draw[gray](3.69,0)--(3.69,-1);
\draw[gray](6,0)--(6,-1);

\draw [very thick] plot [smooth, tension=1] coordinates { (0,-1) (1,0.99) (3,-1.5) (4.8,0.99)(6,-1)};

%\draw[very thick, red] (0,-1) to[out=80,in=180] (1,1);
%\draw[very thick, red] (1,1) to[out=0,in=180] (3,-1.5);

\draw[thick] (-0.5,0)--(6.5,0);
\draw[thick] (0,-2)--(0,2);

\fill[black] (0, 1) circle (2pt);
\draw (-0.2,1) node {$1$};

\fill[black] (0, -1) circle (2pt);
\draw (-0.4,-1) node {$-1$};

\fill[black] (0.95, 1) circle (2pt);
\fill[black] (0.95, 0) circle (2pt);
%\draw (0.95,-0.35) node {$\gamma$};

\fill[black] (4.9, 1) circle (2pt);
\fill[black] (4.9, 0) circle (2pt);

\fill[black] (2.26, -1) circle (2pt);
\fill[black] (2.26, 0) circle (2pt);
\draw (2.26,0.35) node {$c_3$};

\fill[black] (3.69, -1) circle (2pt);
\fill[black] (3.69, 0) circle (2pt);
\draw (3.69,0.35) node {$c_2$};

\fill[black](6,-1) circle (2pt);
\fill[black](6,0) circle (2pt);
\draw (6,0.35) node {$c_1$};
\end{tikzpicture}

		\caption{The graph of $\hat{p}_4(s)$. The parameters are $c_1,c_2\in\{1/\lambda_0, 1/b\}$, $c_3=1/a$.}\label{fig:p4}
	\end{center}
\end{figure}
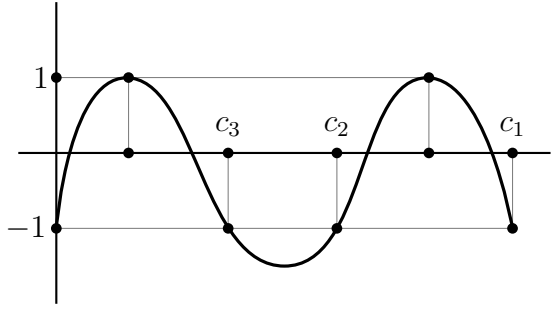

\subsection{$5$-periodic trajectories}

There is a $5$-periodic trajectory of the billiard within $\E$, with non-degenerate caustic $\C_{\lambda_0}$ if and only if the following conditions are satisfied:
\begin{itemize}
	\item the caustics is an ellipse, i.e.~$\lambda_0\in(0,b)$; and
	\item $C_2C_4=C_3^2$.
\end{itemize}

We can calculate:
$$
\begin{aligned}
C_3^2-C_2C_4=&\
\frac{1}{1024 a^5  b^5 \lambda_0 ^7}
\times
\\
&\times
\left(
-(a-b)^6\lambda_0^6
-2 a b (a-b)^2 (a+b) (3 a+b) (a+3 b)\lambda_0^5
\right.
\\
&\qquad\left.
+\ a^2 b^2 (a-b)^2 \left(29 a^2+54 a b+29 b^2\right)\lambda_0^4
-36 a^3 b^3 (a-b)^2 (a+b)\lambda_0^3
\right.
\\
&\qquad\left.
+\ a^4 b^4 \left(9 a^2-34 a b+9 b^2\right)\lambda_0^2
+10 a^5 b^5 (a+b)\lambda_0
-5 a^6 b^6
\right).
\end{aligned}
$$

The winding numbers of $5$-periodic trajectories satisfy $m_0>m_1$, with $m_0=5$ and $m_1$ being even.
Thus, $(m_0,m_1)\in\{(5,2),(5,4)\}$, with $(\tau_1,\tau_2)\in\{(2,1),(0,3)\}$ respecitvely.
The graph of the corresponding polynomial $\hat{p}_5(s)$ is shown in Figures \ref{fig:p52} and \ref{fig:p54}.
\begin{figure}[h]
	\begin{center}
\begin{tikzpicture}

\draw[gray](0,1)--(7.5,1);
\draw[gray](0,-1)--(6.3,-1);

\draw[gray](0.95,0)--(0.95,1);
\draw[gray](4.9,0)--(4.9,1);
\draw[gray](2.26,0)--(2.26,-1);
\draw[gray](3.69,0)--(3.69,-1);
\draw[gray](6.3,0)--(6.3,-1);
\draw[gray](7.5,0)--(7.5,1);

\draw [very thick] plot [smooth, tension=1] coordinates { (0,-1) (1,0.99) (3,-1.5) (4.8,0.99)(6.3,-1)(7.5,1)};

%\draw[very thick, red] (0,-1) to[out=80,in=180] (1,1);
%\draw[very thick, red] (1,1) to[out=0,in=180] (3,-1.5);

\draw[thick] (-0.5,0)--(8,0);
\draw[thick] (0,-2)--(0,2);

\fill[black] (0, 1) circle (2pt);
\draw (-0.2,1) node {$1$};

\fill[black] (0, -1) circle (2pt);
\draw (-0.4,-1) node {$-1$};

\fill[black] (0.95, 1) circle (2pt);
\fill[black] (0.95, 0) circle (2pt);
%\draw (0.95,-0.35) node {$\gamma$};

\fill[black] (4.9, 1) circle (2pt);
\fill[black] (4.9, 0) circle (2pt);

\fill[black] (2.26, -1) circle (2pt);
\fill[black] (2.26, 0) circle (2pt);
\draw (2.26,0.35) node {$c_3$};

\fill[black] (3.69, -1) circle (2pt);
\fill[black] (3.69, 0) circle (2pt);
\draw (3.69,0.35) node {$c_2$};

\fill[black](6.3,-1) circle (2pt);
\fill[black](6.3,0) circle (2pt);

\fill[black] (7.5, 1) circle (2pt);
\fill[black] (7.5, 0) circle (2pt);
\draw (7.5,-0.35) node {$c_1$};
\end{tikzpicture}
		\caption{The graph of $\hat{p}_5(s)$.
			The parameters are $c_1=1/\lambda_0$, $c_2=1/b$, $c_3=1/a$; $(m_0,m_1)=(5,2)$, $(\tau_1,\tau_2)=(2,1)$.}\label{fig:p52}
	\end{center}
\end{figure}
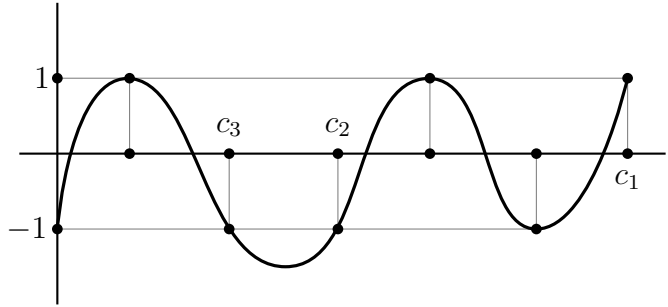

\begin{figure}[h]
	\begin{center}
\begin{tikzpicture}

\draw[gray](0,1)--(8,1);
\draw[gray](0,-1)--(7.35,-1);

\draw[gray](1,0)--(1,1);
\draw[gray](4.75,0)--(4.75,1);
\draw[gray](3,0)--(3,-1);
\draw[gray](6.1,0)--(6.1,-1);
\draw[gray](7.35,0)--(7.35,-1);
\draw[gray](8,0)--(8,1);

\draw [very thick] plot [smooth, tension=1] coordinates { (0,-1) (1,0.99) (3,-0.99) (4.8,0.99)(6.7,-1.5)(8,1)};

%\draw[very thick, red] (0,-1) to[out=80,in=180] (1,1);
%\draw[very thick, red] (1,1) to[out=0,in=180] (3,-1.5);

\draw[thick] (-0.5,0)--(8.5,0);
\draw[thick] (0,-2)--(0,2);

\fill[black] (0, 1) circle (2pt);
\draw (-0.2,1) node {$1$};

\fill[black] (0, -1) circle (2pt);
\draw (-0.4,-1) node {$-1$};

\fill[black] (1, 1) circle (2pt);
\fill[black] (1, 0) circle (2pt);
%\draw (0.95,-0.35) node {$\gamma$};

\fill[black] (4.75, 1) circle (2pt);
\fill[black] (4.75, 0) circle (2pt);

\fill[black] (3, -1) circle (2pt);
\fill[black] (3, 0) circle (2pt);
\

\fill[black](6.1,-1) circle (2pt);
\fill[black](6.1,0) circle (2pt);
\draw (6.1,0.35) node {$c_3$};

\fill[black](7.35,-1) circle (2pt);
\fill[black](7.35,0) circle (2pt);
\draw (7.35,0.35) node {$c_2$};

\fill[black] (8, 1) circle (2pt);
\fill[black] (8, 0) circle (2pt);
\draw (8,-0.35) node {$c_1$};
\end{tikzpicture}
		\caption{The graph of $\hat{p}_5(s)$.
			The parameters are $c_1=1/\lambda_0$, $c_2=1/b$, $c_3=1/a$; $(m_0,m_1)=(5,4)$, $(\tau_1,\tau_2)=(0,3)$.}\label{fig:p54}
	\end{center}
\end{figure}
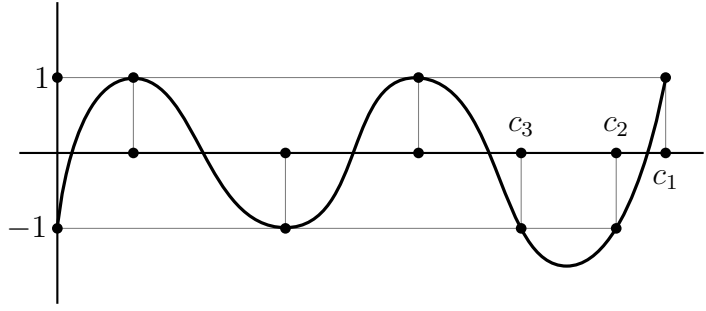

\subsection{$6$-periodic trajectories}

There is a $6$-periodic trajectory of the billiard within $\E$, with non-degenerate caustic $\C_{\lambda_0}$ if and only if $B_4^2-B_3 B_5=0$.

We calculate:
$$
\begin{aligned}
B_4^2-B_3 B_5
=
-\frac{B_2C_2}{256 (ab)^4\lambda_0^{3} }
\ &\times
\left((a-b)(a+3b)\lambda_0 ^2-2 a b (a-b)\lambda_0 +a^2 b^2\right)
\\
&\times
\left(
(a-b)(3a+b)\lambda_0 ^2 -2ab(a-b)\lambda_0  -a^2 b^2
\right).
\end{aligned}
$$

First, let us consider the condition $B_2=0$.
We have
\begin{equation}\label{eq:B2}
B_2=-\frac{a^2 b^2-2 a b \lambda _0 (a+b)+\lambda _0^2 (a-b)^2}{8 (a b \lambda _0)^{3/2}}
\end{equation}
so $B_2=0$ is equivalent to:
$
\lambda_0=ab/(\sqrt{a}\pm\sqrt{b})^2.
$
From the condition $\lambda_0<b$, we have a unique solution which gives an ellipse as the caustic:
$$
\lambda_0=\frac{ab}{(\sqrt{a}+\sqrt{b})^2}.
$$
For $a>4b$, the option $\lambda_0=ab/(\sqrt{a}-\sqrt{b})^2$ will provide a hyperbola as the caustic, see Figure \ref{fig:traj62h}.
\begin{figure}[h]
	\begin{center}
\begin{tikzpicture}[scale=2]

\draw (0,0) ellipse (2.1213cm and 1cm);

\draw [domain=-1:1] plot ({0.9597*sqrt(1+\x*\x/2.5789)}, {\x} );

\draw [domain=-1:1] plot ({-0.9597*sqrt(1+\x*\x/2.5789)}, {\x} );

\draw[thick] (1.03238, 0.873586)--(0.857015, -0.914759)--(-0.207025, 0.995226)--(-1.03238, -0.873586)--(-0.857015, 0.914759)--(0.207025, -0.995226)--(1.03238, 0.873586);

\end{tikzpicture}
		\caption{A $6$-periodic trajectory with a hyperbola as a caustic.
			The parameters are $a=4.5$, $b=1$, $\lambda_0=ab/(\sqrt{a}-\sqrt{b})^2$; $(m_0,m_1)=(6,2)$.
			The trajectory is symmetric with respect to the origin.
		}\label{fig:traj62h}
	\end{center}
\end{figure}
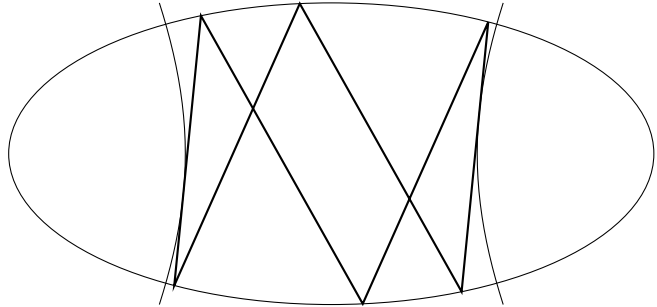

\begin{remark}
	Notice that the winding numbers of $6$-periodic trajectories obtained from the condition $B_2=0$ are $(m_0,m_1)=(6,2)$.
	This is because $B_2=0$ is equivalent to $3P_0\sim P_{\infty}$ on the elliptic curve.
	On the other hand, the obtained trajectories are $3$-periodic in elliptic coordinates, and they transverse the $\mathsf{y}$-axis only once along the period, so the condition of their periodicity is equivalent to:
	$$
	\begin{aligned}
	&3(P_0-P_b) + (P_{\lambda_0}-P_a)\sim0,\quad &\text{for}\quad& \lambda_0\in(b,a),
	\\
	&3(P_0-P_{\lambda_0}) + (P_b-P_a)\sim0,\quad &\text{for}\quad& \lambda_0\in(0,b).
	\end{aligned}
	$$
	Both of the obtained relations are equivalent to $3P_0\sim 3P_{\infty}$, since
	$2P_{\infty}\sim2P_a\sim2P_b\sim2P_{\lambda_0}$ and $3P_{\infty}\sim P_a+P_b+P_{\lambda_0}$.
\end{remark}

Next, $C_2=0$ gives $3$-periodic trajectories.

The discriminant of $(a-b)(a+3b)\lambda_0 ^2-2 a b (a-b)\lambda_0 +a^2 b^2$ is $-16 a^2 b^3 (a-b)$, which is negative, so the expression has no real roots in $\lambda_0$.

On the other hand, $(a-b)(3a+b)\lambda_0 ^2 -2ab(a-b)\lambda_0  -a^2 b^2$ has two real roots:
$$
\lambda_0=\frac{ab(a-b)\pm2ab \sqrt{a(a - b)}}{(a-b)(3a+b)}.
$$
The smaller one is always negative, so it does not correspond to a real billiard trajectory.

Let us analyze the bigger one:
$$
\lambda_0=\frac{ab(a-b)+2ab \sqrt{a(a - b)}}{(a-b)(3a+b)}.
$$
We have that always $\lambda_0>b$, since:
\begin{gather*}
\frac{ab(a-b)+2ab \sqrt{a(a - b)}}{(a-b)(3a+b)}>b
\qquad\Leftrightarrow\qquad
b^2(3a+b)>0.
\end{gather*}
We need to check the condition $\lambda_0<a$, which is equivalent to:
\begin{gather*}
\frac{ab(a-b)+2ab \sqrt{a(a - b)}}{(a-b)(3a+b)}<a
\qquad\Leftrightarrow\qquad
a>\frac43b.
\end{gather*}

Thus for $a>4b/3$, there is another class of $6$-periodic trajectories which have a hyperbola as a caustic, see Figure \ref{fig:traj64h}.
\begin{figure}[h]
	\begin{center}
\begin{tikzpicture}[scale=2]

\draw (0,0) ellipse (1.4142cm and 1cm);

\draw [domain=-0.3:0.3] plot ({0.9519*sqrt(1+\x*\x*10.6569)}, {\x} );

\draw [domain=-0.3:0.3] plot ({-0.9519*sqrt(1+\x*\x*10.6569)}, {\x} );

\draw[thick] (1.2041, 0.52447)--(0.462044, -0.945123)--(-1.32834, 0.343164)--(1.2041, -0.52447)--(0.462044, 0.945123)--(-1.32834, -0.343164)--(1.2041, 0.52447);

\end{tikzpicture}
		\caption{A $6$-periodic trajectory with a hyperbola as a caustic.
			The parameters are $a=2$, $b=1$, $\lambda_0=(ab(a-b)+2ab \sqrt{a(a - b)})/((a-b)(3a+b))$; $(m_0,m_1)=(6,4)$.
			The trajectory is symmetric with respect to the longer axis of the ellipse.
		}\label{fig:traj64h}
	\end{center}
\end{figure}
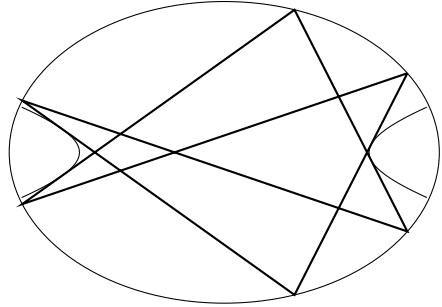

\begin{remark}
	The winding numbers of those $6$-periodic trajectories are $(m_0,m_1)=(6,4)$.
	The obtained trajectories are $3$-periodic in elliptic coordinates, and they transverse the $\mathsf{y}$-axis twice along the period, so the condition of their periodicity is equivalent to:
	$$
	\begin{aligned}
	&3(P_0-P_b) + 2(P_{\lambda_0}-P_a)\sim0,
	\end{aligned}
	$$
	which is equivalent to $3P_0\sim 3P_b$.
\end{remark}

\section{Periodic billiard trajectories and classical extremal polynomials}\label{sec:classicalextrema}
We want to employ the classical theory of extremal polynomials on two intervals
to get explicit formulas for the polynomials $\hat p_n$. As it is well know, the extremal polynomials don't exist on any union of two disjoint real intervals. Thus, we want to exploit the relationship between the end points of such configurations of two intervals and to relate it to the formulas for the caustics which generate periodic billiard trajectories. To illustrate the main idea, we are going to start with families of polynomials, introduced by Zolotarev in 1870's, \cites{AhiezerAPPROX, Akh4}. Later on, we will use two families of polynomials derived by Akhiezer, first in 1928 and second, more general, in 1930's to get the general formulae for our polynomials $\hat p_n$, \cites{Akh1, Akh2, Akh3}.

Let us recall that the celebrated Chebyshev polynomials $T_n(x), n= 0, 1, 2,\dots$
defined by the recursion:
\begin{equation}\label{eq:cheb1}
T_0(x)=1, \, T_1(x)=x,\, T_{n+1}(x)+T_{n-1}(x)=2xT_n(x),
\end{equation}
for $n=1, 2\dots$ can be parameterized as
\begin{equation}\label{eq:cheb2}
T_n(x)=\cos n\phi,\, x=\cos\phi,
\end{equation}
or, alternatively:
\begin{equation}\label{eq:cheb3}
T_n(x)=\frac{1}{2}\left(v^n+\frac{1}{v^n}\right), \quad x=\frac{1}{2}\left(v+\frac{1}{v}\right).
\end{equation}
Denote $L_0=1$ and $L_n=2^{1-n}, n=1, 2,\dots$.
Then the Chebyshev Theorem states that the polynomials $L_nT_n(x)$ are characterized as the solutions of the following minmax problem:

\emph{find the polynomial of degree $n$ with the leading coefficient equal 1 which minimizes the uniform norm on the interval $[-1, 1]$.}

\subsection{Zolotarev polynomials:}

Following the ideas of Chebyshev, his student Zolotarev posed and solved a handful of problems, including the following:

\emph{For the given real parameter $\sigma$ and all polynomials of degree $n$ of the form:
	\begin{equation}\label{eq:zol1}
	p(x)=x^n-n\sigma x^{n-1} + p_2x^{n-2}+\dots p_n,
	\end{equation}
	find the one with the minimal uniform norm on the interval $[-1, 1]$.}

Denote this minimal uniform norm as $L_n=L(\sigma, n)$.

For $\sigma>\tan^2(\Pi/2n)$, the solution $z_n$ has the following property (\cite{AhiezerAPPROX}, p. 298), see Figure \ref{fig:zn}:
\quote{ $\Pi1$ -- The equation $z_n^2(x)=L_n^2$ has $n-2$ double solutions in the open interval $(-1, 1)$ and simple solutions at $-1$, $1$, $\alpha$, $\beta$, where $1<\alpha <\beta$, while in the union of the intervals $[-1,1]\cup [\alpha, \beta]$ the inequality $z_n^2\le L_n^2$ is satisfied and $z_n^2>L_n^2$ in the complement.}
\begin{figure}[h]
	\begin{center}
\begin{tikzpicture}

\draw[gray](0,1)--(8,1);
\draw[gray](0,-1)--(7.35,-1);

\draw[gray](1,0)--(1,1);
\draw[gray](4.73,0)--(4.73,1);
\draw[gray](3,0)--(3,-1);
\draw[gray](6.1,0)--(6.1,-1);
\draw[gray](7.35,0)--(7.35,-1);
\draw[gray](8,0)--(8,1);
\draw[gray] (0,-1)--(0,1);

\draw [very thick] plot [smooth, tension=1] coordinates { (0,-1) (1,0.99) (3,-0.99) (4.8,0.99)(6.7,-1.5)(8,1)};

%\draw[very thick, red] (0,-1) to[out=80,in=180] (1,1);
%\draw[very thick, red] (1,1) to[out=0,in=180] (3,-1.5);

\draw[very thick] (-0.5,0)--(8.5,0);

\fill[black] (0, 0) circle (2pt);
\draw (-0.4,-0.3) node {$-1$};

\fill[black] (0, 1) circle (2pt);
\draw (-0.4,1) node {$L_n$};

\fill[black] (0, -1) circle (2pt);
\draw (-0.6,-1) node {$-L_n$};

\fill[black] (1, 1) circle (2pt);
\fill[black] (1, 0) circle (2pt);
%\draw (0.95,-0.35) node {$\gamma$};

\fill[black] (4.73, 1) circle (2pt);
\fill[black] (4.73, 0) circle (2pt);

\fill[black] (3, -1) circle (2pt);
\fill[black] (3, 0) circle (2pt);
\

\fill[black](6.1,-1) circle (2pt);
\fill[black](6.1,0) circle (2pt);
\draw (6.1,0.35) node {$1$};

\fill[black](7.35,-1) circle (2pt);
\fill[black](7.35,0) circle (2pt);
\draw (7.35,0.35) node {$\alpha$};

\fill[black] (8, 1) circle (2pt);
\fill[black] (8, 0) circle (2pt);
\draw (8,-0.35) node {$\beta$};
\end{tikzpicture}
		\caption{The graph of $z_n(x)$.}\label{fig:zn}
	\end{center}
\end{figure}
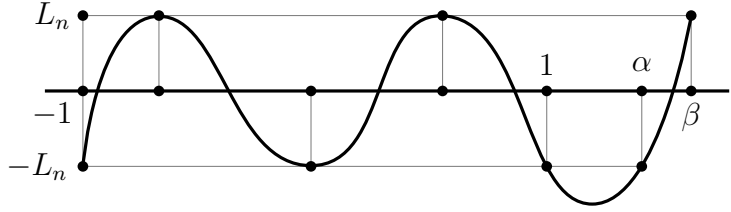

The polynomials $z_n$ are given by the following explicit formulae:
\begin{equation}\label{eq:zn}
z_n=\ell_n\left(v(u)^n+\frac1{v(u)^n}\right), \, x=\frac{\sn^2u +\sn^2\frac{K}{n}}{\sn^2u -\sn^2\frac{K}{n}},
\end{equation}
where
$$\ell_n=\frac1{2^n}\left(\frac{\sqrt{\kappa}\theta_1^2(0)}{H_1\left(\frac{K}{n}\right)\theta_1\left(\frac{K}{n}\right)}\right)^{2n}, \quad v(u)=\frac{H\left(\frac{K}{n}-u\right)}{H\left(\frac{K}{n}+u\right)}$$
and
$$\sigma=\frac{2\sn\frac{K}{n}}{\cn\frac{K}{n}\dn\frac{K}{n}}\left(\frac1{\sn\frac{2K}{n}}-\frac {\theta'\left(\frac{K}{n}\right)}{\theta\left(\frac{K}{n}\right)}\right)-1.
$$
Formulae for the endpoints of the second interval are:
\begin{equation}\label{eq:alphabetan}
\alpha =\frac{1+\kappa^2\sn^2\frac{K}{n}}{\dn^2\frac{K}{n}}, \quad \beta =\frac{1+\sn^2\frac{K}{n}}{\cn^2\frac{K}{n}},
\end{equation}
with
$$
\kappa^2=\frac{(\alpha-1)(\beta+1)}{(\alpha+1)(\beta-1)}.
$$
In order to derive the formulas for $\hat p_3$ in terms of $z_3$, let us construct
an affine transformation:
$$
h:[-1, 1]\cup [\alpha, \beta]\rightarrow [0, a^{-1}]\cup[b^{-1}, \lambda_0^{-1}], \, h(x) = \hat a x +\hat b.
$$
According to Cayley's condition (\ref{eq:lambda3})
$$
\lambda_0=-\frac{ab(a+b)+2ab\sqrt{a^2-ab+b^2}}{(a-b)^2}.
$$
We immediately get
$$
\hat a = \hat b, \, \hat a = \frac1{2a}
$$
and
\begin{gather}\label{eq:alphabetacayley3}
\alpha = 2t-1,\\
\beta = -\frac{2t^2-3t+3+2\sqrt{t^2-t+1}}{t+1+2\sqrt{t^2-t+1}},\\
\beta = \frac{4}{3} \sqrt{t^2-t+1} +\frac {2}{3}t +\frac {5}{3},
\end{gather}
where $t=a/b$.

There is a relation between $\alpha$ and  $\beta$ as defined by the formulae above:
\begin{equation}\label{eq:criterion3}
9\beta^2-3\alpha^2-6\alpha\beta+12\alpha - 36 \beta +56=0.
\end{equation}
Now we get the following
\begin{proposition} The polynomial $\hat p_3$ can be expressed through the Zolotarev polynoamil $z_3$ up to a nonessential constant factor:
	$$
	\hat p_3 (s) \sim z_3( 2as-1).
	$$
\end{proposition}

To verify the proposition, we should  certify that $\alpha$ and $\beta$ defined in (\ref{eq:alphabetan}) for $n=3$ satisfy the relation (\ref{eq:alphabetacayley3}).

In order to do that we will use well-known identities for the Jacobi elliptic functions:

\begin{gather}
\sn^2u+\cn^2u+1,
\\
\kappa^2\sn^2u+\dn^2u=1,
\\
\sn(u+v)=\frac{\sn\, u\cn\, v\dn \,v + \sn\, v\cn\, u\dn\, u}{1-\kappa^2\sn^2 u \sn ^2 v},
\\
\sn (K-u)=\frac{\cn\, u}{\dn\, u}.
\end{gather}
In particular, we get
\begin{gather}
\sn \left(\frac {2K}{3}\right)= \frac{2\sn\, \frac{K}{3}\cn \,\frac{K}{3}\dn \,\frac{K}{3}}{1-\kappa^2\sn^4\frac{K}{3}},
\\
\sn \left(\frac {2K}{3}\right)= \sn \left(K-\frac {K}{3}\right)=\frac{\cn \,\frac{K}{3}}{\dn\, \frac{K}{3}}.
\end{gather}
Let us denote
$$Y=\sn \left(\frac {K}{3}\right),$$
then from the previous two relations we get:
$$
1-2Y+2\kappa^2Y^3 -\kappa^2Y^4=0.
$$
We can express $\kappa$ in terms of $Y$ and get:
$$
\kappa^2=\frac{2Y-1}{Y^3(2-Y)}.
$$
By plugging the last relation into (\ref{eq:alphabetan}) for $n=3$ we get
$$
\alpha=\frac{Y^2-4Y+1}{Y^2-1}.
$$
Since, at the same time from the Cayley condition we have $\alpha = 2t-1$, with $t=a/b$, we can express $Y$ in terms of $t$:
$$
(t-1)Y^2+2Y-t=0,
$$
and
\begin{equation}\label{eq:Yt}
Y=\frac{-1\pm \sqrt{1-t+t^2}}{t-1}.
\end{equation}
We plug the last relation into the formula for $\beta$ from (\ref{eq:alphabetan}) for $n=3$ and we get a formula for $\beta$ in terms of $t$:
\begin{gather}\label{eq:betat}
\beta=\frac{(t-1)^2+(-1\pm\sqrt{t^2-t+1})^2}{(t-1)^2-(-1\pm\sqrt{t^2-t+1})^2},
\\
\beta = -\frac{2t^2-3t+3\pm2\sqrt{t^2-t+1}}{t+1\pm 2\sqrt{t^2-t+1}}.
\end{gather}
We see that the last formula with the choice of the $+$ sign corresponds to a formula for $\beta$ from (\ref{eq:alphabetacayley3}). This finalizes the verification.
One can observe that the $-$ sign option from the formula above would correspond to
the $-$ sign in the formula for $\lambda_0$ above the formula (\ref{eq:lambda3}).

Among the polynomials $\hat p_n$ the property of type $\Pi1$ can be attributed
only to those with $n=2k+1$ and winding numbers $(2k+1, 2k)$, in other words to those with the signature $(0, 2k-1)$. For example, this is satisfied for the polynomial $\hat p_5$ presented in the Fig. \ref{fig:p54} while it is not true for the polynomial $\hat p_5$ presented in the Fig. \ref{fig:p52}.

\subsection{Akhiezer polynomials on symmetric intervals $[-1, -\alpha]\cup[\alpha, 1]$}

The problem of finding polynomials of degree $n$ with the leading coefficient 1 and minimizing the uniform norm on the union of two symmetric intervals $[-1, -\alpha]\cup[\alpha, 1]$, for given $0<\alpha <1$ appeared to be of a significant interest in radio-techniques applications. Following the ideas of Chebyshev and Zolotarev, Akhiezer derived in 1928 the explicit formulae for such polynomials $A_n(x;\alpha)$ with the deviation $L_n(\alpha)$, \cite{AhiezerAPPROX, Akh4}.

These formulas are specially simple in the case of even degrees $n=2m$, when Akhiezer polynomials $A_{2m}$ are obtained by a quadratic substitution from the Chebyshev polynomial $T_m$:

\begin{equation}\label{eq:A2m}
A_{2m}(x;\alpha)=\frac{(1-\alpha^2)^m}{2^{2m-1}}T_m\left(\frac{2x^2-1-\alpha^2}
{1-\alpha^2}\right),
\end{equation}
with
$$
L_{2m}(\alpha)=\frac{(1-\alpha^2)^m}{2^{2m-1}}.
$$
We are going to construct $\hat p_4(s)$ up to a nonessential constant factor as a composition of $A_4(x;\alpha)$ for certain $\alpha$ and an affine transformation.
We are going to study the possibility to have an affine transformation
$$
g:[-1,-\alpha]\cup[\alpha, 1]\rightarrow [0, c_3]\cup [c_2, c_1],\quad g(x)=\hat a x +\hat b
$$
in two versions, depending if the caustic corresponds to the reciprocical value of $c_1$ or of $c_2$. The former case corresponds to the case of caustic being ellipse and the latter of being hyperbola. This we will denote these two cases (E) and (H) respectively.

{\bf Case (E).}

For
$$
g:[-1,-\alpha]\cup[\alpha, 1]\rightarrow [0, a^{-1}]\cup [b^{-1}, \lambda^{-1}],\quad g(x)=\hat a x +\hat b
$$
we get
$$
\hat a =\hat b,\quad \alpha =\frac {a-b}{a+b}, \quad \hat a= \frac {a+b}{2ab}.
$$
Thus:
$$
g(1)=2\hat a= \frac{a+b}{ab}=\frac1{\lambda}
$$
implies
$$
\lambda=\frac{ab}{a+b},
$$
which coincides with the formula for the caustic -ellipse for $4$ periodic billiard trajectories derived from the Cayley condition.
\begin{proposition} In this case the polynomial $\hat p_4(s)$ is equal up to a constant multiplier to
	\begin{equation}\label{eq:p4t2}
	\hat p_4(s) ~\sim \frac{2a^2b^2}{(a+b)^4}T_2(2abs^2 -2(a+b)s+1),
	\end{equation}
	where $T_2(x)=2x^2-1$ is the second Chebyshev polynomial.
\end{proposition}

{\bf Case (H).}

For
$$
g:[-1,-\alpha]\cup[\alpha, 1]\rightarrow [0, a^{-1}]\cup [\lambda^{-1}, b^{-1}],\quad g(x)=\hat a x +\hat b
$$
we get
$$
\hat a =\hat b,\quad \alpha =\frac {a-2b}{a}, \quad \hat a= \frac {1}{2b}.
$$
Thus:
$$
g(\alpha)=\hat a(1+\alpha)= \frac{a-b}{ab}=\frac1{\lambda}
$$
implies
$$
\lambda=\frac{ab}{a-b},
$$
which coincides with the formula for the caustic - hyperbola for $4$ periodic billiard trajectories derived from the Cayley condition.
\begin{proposition} In this case the polynomial $\hat p_4(s)$ is equal up to a constant multiplier to
	\begin{equation}\label{eq:p4t2bis}
	\hat p_4(s) \sim \frac{2b^2(a-b)^2}{a^4}T_2\left(\frac{a^2s^2-4a^2bs +8b^3(a-b)}{8b^3(a -b)}\right),
	\end{equation}
	where $T_2(x)=2x^2-1$ is the second Chebyshev polynomial.
\end{proposition}

The polynomials $\hat p_4$ from (\ref{eq:p4t2}) and  (\ref{eq:p4t2bis}) can be rewritten in the canonical form respectively as:
$$
\hat p_4 (s) \sim  8a^2b^2 s^4 -16ab(a+b)s^3 + 8 (a^2+3ab +b^2) s^2-8(a+b) s+1
$$
and

$$
\hat p_4 (s) \sim a^4s^4-8a^4bs^3+16a^2b^2(a^2+ab-b^2)s^2 +64 a^2b^4(b-a)s +32b^9(b-2a).
$$

By analyzing the structure of the extremal points of the polynomials $A_n$, we generalize  the last two propositions as follows:
\begin{proposition} For trajectories with period $4k$ and winding numbers $(4k, 2k)$,
	the corresponding polynomials $\hat p_{4k}$ are equal up to a nonessential constant factor to:
	$$
	\hat p_{4k}(s) ~\sim T_{2k}(2abs^2 -2(a+b)s+1),
	$$
	for the caustic being ellipse, and to
	$$
	\hat p_{4k}(s) \sim T_{2k}\left(\frac{a^2s^2-4a^2bs +8b^3(a-b)}{8b^3(a -b)}\right),
	$$
	in the case of hyperbola as the caustic,
	where $T_{2k}$ is the $2k$-th Chebyshev polynomial.
\end{proposition}

Since the polynomials $A_n$ have symmetrically distributed extremal points in the intervals $[-1,-\alpha]$ and $[\alpha, 1]$, they can't serve as models for polynomials $\hat p_k$ except in the cases listed in the above Proposition.

\subsection{General Akhiezer polynomials on unions of two intervals}

So far, we managed to express polynomials $\hat p_3, \hat p_4$,  $\hat p_5$ only in the case of the signature $(0,3)$, $\hat p_{2n+1}$ with the signature $(0, 2n-1)$ and $\hat p_{4k}$ only in the case of the signature $(2k-1, 2k-1)$, by use of polynomials of Zolotarev and Akhiezer. However, we were unable to get for example the polynomial $\hat p_5$ with the signature $(2,1)$, as well as polynomials $\hat p_{4k+2}$ and $\hat p_{4k}, \hat p_{2k+1}$ with signatures different that those listed above.
Thus, in order to get the general formulae for the general polynomials $\hat p_n$, we need to employ a more general theory of extremal polynomials on two intervals, developed by Akhiezer \cite{Akh1, Akh2, Akh3}.

Following Akhiezer, let us consider the union of two intervals
$
[-1, \alpha]\cup[\beta, 1],
$
where
$$
\alpha = 1-2\sn^2\frac {m}{n}K, \quad \beta = 2\sn^2\frac {n-m}{n}K-1.
$$
Define
\begin{equation}\label{eq:genakhiezer}
TA_n(x, m, k)=L\left(v(u)^n + \frac1{v(u)^n}\right),
\end{equation}
where
$$
v(u)=\frac{H\left(u-\frac{mK}{n}\right)}{H\left(\frac{mK}{n}+u\right)},
$$
$$
x=\frac{\sn^2u\quad\cn^2\frac{m}{n}K+\cn^2u\quad\sn^2\frac{m}{n}K}{\sn^2u-\sn^2\frac{m}{n}K},
$$
and
$$
L=\frac1{2^{n-1}}\left(\frac{\theta(0)\theta_1(0)}
{\theta(\frac{m}{n}K)\theta_1(\frac{m}{n}K)}\right),\quad k^2=\frac{2(\beta-\alpha)}{(1-\alpha)(1-\beta)}.
$$
Akhiezer proved the following results:

\begin{theorem}[Akhiezer]\label{th:akhiezer}
	\begin{itemize}
		\item [(a)] The function $TA_n(x, m, k)$ is a polynomial of degree $n$ in $x$ with the leading coefficient 1 and the second coefficient equal to $-n\tau_1$, where
		$$ \tau_1=-1 + 2\frac{\sn\frac{m}{n}K\quad\cn{\frac{m}{n}K}}{\dn{\frac{m}{n}K}}
		\left(\frac1{\sn\frac{2m}{n}K}-\frac{\theta'(\frac{m}{n}K)}{\theta(\frac{m}{n}K)}\right).
		$$
		\item [(b)] The maximum of the modulus of $T_n$ on the union of the two intervals
		$[-1, \alpha]\cup[\beta, 1]$ is $L$.
		\item[(c)] The function $T_n$ takes the values $\pm L$ with alternating signs at $\mu=n-m+1$ consecutive points of the interval $[-1, \alpha]$ and at $\nu=m+1$ consecutive points of the interval $[\beta, 1]$. In addition
		$$
		T_n(\alpha, m, k)=T_n(\beta, m, k)=(-1)^mL,
		$$
		and for any $x\in (\alpha, \beta)$ it holds:
		$$
		(-1)^mT_n(x, m, k)>L.
		$$
		\item [(d)] Let $F$ be a polynomial of degree $n$ in $x$ with the leading coefficient equal 1, such that:
		\begin{itemize}
			\item[i)] $\max |F(x)| = L$ for $x\in [-1, \alpha]\cup[\beta, 1]$;
			\item[ii)] $F(x)$ takes values $\pm L$ with alternating signs at $n-m+1$ consecutive points of the interval $[-1, \alpha]$ and at $m+1$ consecutive points of the interval $[\beta, 1]$.
		\end{itemize}
		Then $F(x)=T_n(x, m, k)$.
	\end{itemize}
\end{theorem}

Let us observe that the polynomials $\hat p_n(s)$ satisfy the conditions of the item (d) of the Akhiezer Theorem, up to an affine change of variables and up to a nonessential constant factor. Indeed, let us denote by
$$\hat p_n(s, 2l)$$
such a polynomial $\hat p_n$ which corresponds to the winding numbers $(n, 2l), 2l<n$. Then the corresponding signature is $(n-2l-1, 2l-1)$. The number of alternating points on the interval $[0, c_3]$ of the polynomial $\hat p_n(s, 2l)$ is equal to $2l+1$, while the number of its alternating points on the interval $[c_2, c_1]$ is equal to $n-2l+1$. Thus $m$ from the Akhiezer Theorem is
$$
m=n-2l.
$$
Now, let us determine the affine transformations in two cases: (E) when the caustic is an ellipse and (H) when the caustic is hyperbola.

{\bf Case (E)}.

For
$$
h:[-1,\alpha]\cup[\beta, 1]\rightarrow [0, a^{-1}]\cup [b^{-1}, \lambda^{-1}],\quad h(x)=\hat a x +\hat b
$$
we get
$$
\hat a =\hat b,\quad  \hat a= \frac {1}{a(\alpha+1)},\quad \frac {\beta+1}{\alpha +1}=\frac{a}{b}.
$$
Thus:
$$
\lambda=\frac{a(\alpha+1)}{2}= \frac{b(\beta+1)}{2}.
$$
We have proved the following theorem.
\begin{theorem} The polynomials $\hat p_n$ can be expressed up to a nonessential multiplier as a composition of a $TA_n$ polynomial and an affine transformation:
	\begin{equation}\label{eq:pntan2}
	\hat p_n(s, 2l)\sim TA_n(a(\alpha+1)s-1; n-2l, k).
	\end{equation}
\end{theorem}

{\bf Case (H)}.

For $n$ even there is one more option, with the caustic being a hyperbola.

For
$$
h_1:[-1,\alpha]\cup[\beta, 1]\rightarrow [0, a^{-1}]\cup [\lambda^{-1}, b^{-1}],\quad h_1(x)=\hat a x +\hat b
$$
we get
$$
\hat a =\hat b,\quad  \hat a= \frac {1}{2b}, \quad \alpha =\frac{2b}{a}-1.
$$
Thus:
$$
\lambda=\frac{2b}{(\beta+1)}.
$$

\begin{theorem} The polynomials $\hat p_n$ can be expressed up to a nonessential multiplier as a composition of a $TA_n$ polynomial and an affine transformation:
	\begin{equation}\label{eq:pntan}
	\hat p_n(s, 2l)\sim TA_n(2bs-1; n-2l, k).
	\end{equation}
\end{theorem}

The relation between $\alpha$ and $\beta$ given by:
$$
\alpha = 1-2\sn^2\frac {m}{n}K, \quad \beta = 2\sn^2\frac {n-m}{n}K-1,
$$
can be seen more clearly if we introduce
$$Z=\sn\frac{mK}{n}.$$
We have then
$$
\sn^2\left(\frac{n-m}{n}K\right)=\frac{1-Z^2}{1-\kappa^2Z^2},
$$
and
$$
\alpha =1-2Z^2, \quad \beta=\frac{1+(\kappa^2-2)Z^2}{1-\kappa^2Z^2}.
$$

\begin{example} Let us illustrate the last two Theorems for $n=4$. Then
	$l=1,\quad m=n-2l=2$.
	It is well-known that
	$$
	\sn^2\frac{K}{2}=\frac{1}{1+\sqrt{1-\kappa^2}}.
	$$
	We have
	$$
	\alpha=-\beta=1-2\sn^2\frac{K}{2}.
	$$
	
	{\bf Case (E)}.
	
	Plugging $\beta=-\alpha$ into
	$$
	\frac {\beta+1}{\alpha +1}=\frac{a}{b}=t,
	$$
	we get
	$$
	\alpha=\frac{1-t}{1+t}.
	$$
	From
	$$
	\lambda=\frac{a(\alpha+1)}{2}=\frac{a}{1+t},
	$$
	we get
	$$
	\lambda=\frac{ab}{a+b},
	$$
	which coincides with one of the values obtained from the Cayley-type condition in (\ref{eq:c4lambda}).
	From the relation
	$$
	\alpha= 1- 2 \sn^2\frac{K}{2}=\frac{1-t}{1+t},
	$$
	we get a relation between $t$, the ratio of the squares of the semi-axes of the ellipse and the elliptic modulus $\kappa$:
	$$
	t=\frac{1}{\sqrt{1-\kappa^2}}.
	$$
	
	Finally, we get:
	$$
	\hat p_4(s, 2)\sim TA_4\left(2\frac{ab}{a+b}s-1; 2, \sqrt{\frac{a^2-b^2}{a^2}}\right).
	$$
	
	{\bf Case (H)}.
	
	From
	$$\beta=-\alpha=1-\frac{2}{t}$$
	we get
	$$
	\lambda=\frac{bt}{t-1},
	$$
	which gives
	$$
	\lambda=\frac{ab}{a-b},
	$$
	which is again one of the values obtained from the Cayley-type condition in (\ref{eq:c4lambda}).

	From the relation
	$$
	\alpha= 1- 2 \sn^2\frac{K}{2}=\frac{2}{t}-1,
	$$
	we get a relation between $t$, the ratio of the squares of the semi-axes of the ellipse and the elliptic modulus $\kappa$:
	$$
	t=\frac{1+\sqrt{1-\kappa^2}}{\sqrt{1-\kappa^2}}.
	$$
	
	Finally, we get:
	$$
	\hat p_4(s, 2)\sim TA_4\left(2bs-1; 2, \sqrt{\frac{a^2-2ab}{(a-b)^2}}\right).
	$$
	
\end{example}

\begin{example} Let us consider now the case $n=5$.
	Let us denote
	$$
	\sn\quad \frac{K}{5}=Y.
	$$
	Then, from the addition formulas we get:
	$$
	\sn^2\quad \frac{2K}{5}=\frac{4Y^2(1-Y^2)(1-\kappa^2Y^2)}{1-\kappa^2Y^4}.
	$$
	Similarly, we get
	$$
	\sn^2 \frac{4K}{5}=\frac{16(1-\kappa^2Y^4)^8Y^2(1-Y^2)(1-\kappa^2Y^2)(1-\sn^2\quad \frac{2K}{5})(1-\kappa^2\sn^2\quad \frac{2K}{5})}{(1-\kappa^2Y^4)^2((1-\kappa^2Y^4)^4-16Y^4(1-Y^2)^2(1-\kappa^2Y^2)^2)^2}.
	$$
	From the last relation and
	$$
	\sn^2 \frac{4K}{5}=\frac{1-Y^2}{1-\kappa^2Y^2}
	$$
	we get:
	$$
	0=P(Z, s)=\sum_{p, q}P_{p,q}Z^ps^q=\sum_{p=0}^{16}E_p(s)Z^{p}=\sum_{q=0}^{8}F_q(Z)s^q,
	$$
	where
	$$Z=Y^2,\quad s=\kappa^2$$
	and
	$$
	\begin{aligned}
	&F_0=-1+16Z-64Z^2+64 Z^3,
	\\
	& F_1=-56Z^2+352Z^3-416Z^4,
	\\
	&F_2=144Z^3-1244Z^4+2160Z^5-1280Z^6+896Z^7-256Z^8,
	\\
	&F_3 =-160Z^4+2144Z^5-4744Z^6+4160Z^7-3264Z^8+1024Z^9,
	\\
	& F_4=64Z^5-1984Z^6+5360Z^7-5830Z^8+5360Z^9-1984Z^{10}+64Z^{11},
	\\
	& F_5 =-160Z^{12}+2144Z^{11}-4744Z^{10}+4160Z^9-3264Z^8+1024Z^7,
	\\
	& F_6=144Z^{13}-1244Z^{12}+2160Z^{11}-1280Z^{10}+896Z^{9}-256Z^8,
	\\
	& F_7=  -56Z^{14}+352Z^{13}-416Z^{12},
	\\
	&F_8= -Z^{16}+16Z^{15}-64Z^{14}+64 Z^{13}.
	\end{aligned}
	$$
	
	Observe the symmetry:
	\begin{equation}\label{eq:summP}
	P_{p,q}=P_{16-p, 8-q}.
	\end{equation}

	{\bf Case $l=1$.}
	From the Akhiezer formulae we get
	
	$$
	\alpha = 1-2\sn^2\frac {3}{5}K, \quad \beta = 2\sn^2\frac {2}{5}K-1,
	$$
	and
	\begin{equation}\label{eq:tab52}
	\frac{a}{b}=t=\frac{\sn^2\frac {2}{5}K}{1-\sn^2\frac {3}{5}K}.
	\end{equation}
	Thus,
	$$
	\lambda_0=a\left(1-\sn^2\frac {3}{5}K\right),
	$$
	and
	$$
	\hat p_5(s, 2)\sim TA_4\left(2a\left(1-\sn^2\frac {3}{5}K\right)s-1; 3, \kappa\right).
	$$
	
	From the equation (\ref{eq:tab52}) and one of the addition formulas:
	$$
	\sn^2 \frac {3}{5}K=\frac{1-\sn^2\frac{2}{5}K}{1-\kappa^2\sn^2\frac{2}{5}K}
	$$
	we get:
	$$
	\kappa^2=\frac{t-1}{Wt},
	$$
	where $W=\sn^2 \frac{3}{5}K.$
	
	{\bf Case $l=2$.}
	
	Similarly, from
	$$
	\alpha = 1-2\sn^2\frac {1}{5}K, \quad \beta = 2\sn^2\frac {4}{5}K-1,
	$$
	and
	$$
	\frac{a}{b}=t=\frac{\sn^2\frac {4}{5}K}{1-\sn^2\frac {1}{5}K}.
	$$
	we get
	$$
	\lambda_0=a\left(1-\sn^2\frac {1}{5}K\right),
	$$
	and
	$$
	\hat p_5(s, 4)\sim TA_4\left(2a\left(1-\sn^2\frac {1}{5}K\right)s-1; 1, \kappa\right).
	$$
	Here
	$$
	\kappa^2=\frac{t-1}{W_1},
	$$
	where $W_1=\sn^2\frac{4}{5}K.$

\end{example} 

\section{Discriminantly separable and discriminantly factorizable polynomials}\label{sec:separable}

\subsection{Definition of discriminantly separable polynomials}

We briefly review the basic notions and indicate several relationships and applications to different areas of mathematics and mechanics,
we provide a general definition of the discriminantly separable
polynomials. By $\mathcal{P}_m^n$ denote the polynomials of $m$
variables of degree $n$ in each variable.
\begin{definition} [\cite{Drag2010}] A polynomial $F(x_1,\dots,x_n)$ is
	\emph{discriminantly separable} if there exist polynomials $f_1(x_1)$, \dots, $f_n(x_n)$ such that the discriminant $\mathcal{D}_{x_i}F$ of $F$ with respect to $x_i$ satisfies: 
	\[
	\mathcal{D}_{x_i}F(x_1,\dots, \hat x_i, \dots, x_n)=\prod_{j\ne
		i}f_j(x_j),
	\]
	for each $i=1,\dots,n$.
	$F$ is \emph{symmetrically discriminantly separable} if
	$f_2=f_3=\dots = f_n$,
	while it is \emph{strongly discriminantly separable} if
	$f_1=f_2=f_3=\dots = f_n.$ It is \emph{weakly discriminantly
		separable} if there exist polynomials $f^j_i(x_i)$ such that for
	every $i=1,\dots , n$
	\[
\mathcal{D}_{x_i}F(x_1,\dots, \hat x_i, \dots, x_n)=\prod_{j\ne
		i}f^i_j(x_j).
	\]
\end{definition}

\subsection{$n$-valued groups}\label{sec:dvg}
\medskip

The idea of $n$-valued groups, on a local level,  goes back to
Buchstaber and Novikov, to  their 1971 study of
characteristic classes of vector bundles. That concept was
significantly developed further by Buchstaber and his collaborators (\cite{Buch2006} and references therein).
An $n$-valued group on $X$ can be defined as a map:
$$
\aligned
&m:\quad X\times X \rightarrow (X)^n\\
&m(x,y)=x*y=[z_1,\dots, z_n],
\endaligned
$$
where $(X)^n$ denotes the symmetric $n$-th power of $X$ and $z_i$
coordinates therein. Such a map should satisfy the following axioms.
\begin{description}
\item[Associativity:]  The condition of equality of two $n^2$-sets
$$
\aligned &[x*(y*z)_1,\dots, x*(y*z)_n]\\
&[(x*y)_1*z,\dots, (x*y)_n*z]
\endaligned
$$
for all triplets $(x,y,z)\in X^3$. 

\item[Unit element:]
An element $e\in X$ is \emph{a unit} if
$
e*x=x*e=[x,\dots,x],
$
for all $x\in X$. 
\item[Inverse:]
A map $\inv: X\rightarrow X$ is \emph{an inverse} if 
$e\in \inv(x)*x$ and $e\in x*\inv(x)$,
for all $x\in X$. 
\end{description}
Buchstaber  says that $m$ defines \emph{an $n$-valued
	group structure} $(X, m, e, \inv)$ if it is associative, with a unit
and an inverse.

An $n$-valued group $X$ acts on a set $Y$ if there is a mapping
$$
\phi\ :\ X\times Y \rightarrow (Y)^n,
\quad
\phi (x,y)=x\circ y,
$$
such that the two $n^2$-multisubsets
$x_1\circ (x_2\circ y)$ and $(x_1*x_2)\circ y$
of $Y$ are equal for all $x_1, x_2\in X$, $y\in Y$. It is also
assumed
$
e\circ y=[y,\dots, y]
$
for all $y\in Y$.
\medskip

The list of elementary $n$-valued groups has been done in \cite{Buch2006}.
For a fixed $n$, the corresponding $n$-valued group is
defined by a symmetric polynomial $p_n \in \mathcal P_3^n$.

We recall that the elementary symmetric functions of three variables are denoted as
$s_1, s_2, s_3$:
$$
s_1=x+y+z, \quad s_2=xy+xz+yz, \quad s_3=xyz.
$$

Let us consider a few simple examples. 

\begin{example}\label{lemma:p2}
	Two-valued group $p_2$ is defined by the relation
\begin{gather*}
m_2:\quad \mathbf C \times \mathbf C \rightarrow (\mathbf
C)^2,\\
x *_2 y =[(\sqrt{x}+\sqrt{y})^2, (\sqrt{x}-\sqrt{y})^2].
\end{gather*}
	The product $x *_2 y$ is given by the solutions  of the
	polynomial equation
	$
	p_2(z, x, y)=0
	$
	in $z$,  where
	$$
	p_2(z, x, y)= (x+y+z)^2-4(xy+yz+zx).
	$$
	As observed in \cite{Drag2010},  the polynomial $p_2(z, x,
	y)$ is strongly discriminantly separable:
	$$
	\mathcal D_z(p_2)(x,y)=2x\cdot2y, \quad \mathcal
	D_x(p_2)(y,z)=2y\cdot2z, \quad  \mathcal D_y(p_2)(x,z)=2x\cdot2z,
	$$
so it generates a case
	of generalized Kowalevski system of differential equations from
	\cite{Drag2010}.
\end{example}

Now, we can reproduce a small mathematical experiment from \cite{Drag2012}  with the next
cases of elementary $n$ valued groups, with small $n$.

\begin{example}[$n=3$]\label{example:p3}
	$p_3= s_1^3 - 3^3 s_3$, $\mathcal D_z p_3 = y^2x^2(x-y)^2$.
\end{example}

\begin{example}[$n=4$]\label{example:p4}
$$
	p_4 = s_1^4 - 2^3 s_1^2s_2 + 2^4s_2^2 - 2^7 s_1s_3,\quad
	\mathcal D_z p_4 = y^3x^3(x-y)^2(y+4x)^2(4y+x)^2.
	$$
\end{example}
\begin{example}[$n=5$]\label{example:p5}
$$
	p_5 =s_1^5 - 5^4 s_1^2s_3 + 5^5s_2s_3,\quad
	\mathcal D_zp_5 = y^4x^4(x-y)^4(x^2-y^2- 11xy )^2( x^2-y^2+11xy )^2.
	$$
\end{example}

We see that the polynomials $p_3$, $p_4$, $p_5$ are not any more discriminantly separable.
Nevertheless, following \cite{Drag2012}, we can observe amazing factoriziblity property
of their discriminants. Unexpectedly, the Hadamard-Hankel determinants which appeared in the study of the Cayley-type conditions in previous sections, have the same algebraic properties.

\subsection{Cayley-type conditions and discriminantly factorizible polynomials}

\begin{example}\label{example:c1}
	Denote the numerator in the expression \eqref{eq:C2} as $F_2(\lambda_0,a,b)$:
	$$
	F_2(\lambda_0,a,b)=\lambda_0 ^2 (a-b)^2+2\lambda_0 ab(a+b)-3a^2 b^2.
	$$
	$F_2$ is a discriminantly factorizible  polynomial.
	Its discriminant with respect to $\lambda_0$ is:
	$$
	\mathcal D_{\lambda_0}F_2=16a^2b^2(a^2-ab+b^2).
	$$
\end{example}

\begin{example}\label{example:c2}
	Denote the numerator in \eqref{eq:B3} as $F_3(\lambda_0,a,b)$:
	$$
	\begin{aligned}
	F_3(\lambda_0,a,b)\ &=(a \lambda_0 -a b+\lambda_0  b) (a \lambda_0 +a b-\lambda_0  b) (-a \lambda_0 +a b+\lambda_0  b)
	\\
	&=
	-(a-b)^2(a+b)\lambda_0 ^3 +ab(a-b)^2\lambda_0 ^2+a^2b^2(a+b) \lambda_0
	-a^3 b^3.
	\end{aligned}
	$$
	Its discriminant with respect to $\lambda_0$ is:
	$$
	\mathcal D_{\lambda_0}F_3=64 a^8 b^8 (a-b)^2.
	$$
\end{example}

\begin{example}\label{example:c3}
	The discriminant with respect to $\lambda_0$ of the polynomial numerator of $B_4^2-B_3 B_5$ equals:
	$$
	-309485009821345068724781056\cdot a^{74} b^{74} (a-b)^{18} \left(a^2-a b+b^2\right).
	$$
\end{example}

\begin{example}\label{example:c4}
	The discriminant of the numerator of $C_3^2-C_2C_4$ with respect to $\lambda_0$ is:
	\begin{gather*}
	-87960930222080\cdot a^{38} b^{38} (a-b)^8 
	\times\\ \times
	\left(27 a^6-81 a^5 b+322 a^4 b^2-509 a^3 b^3+322 a^2 b^4-81 a b^5+27 b^6\right).
	\end{gather*}
\end{example}

\begin{example}\label{example:c5}
	Denote the numerator in the expression \eqref{eq:B2} as $F_2(\lambda_0,a,b)$:
	$$
	F_2(\lambda_0,a,b)=\lambda_0 ^2 (a-b)^2-2\lambda_0 ab(a+b)+a^2 b^2.
	$$
	$F_2$ is a strongly discriminantly separable polynomial.
	Its discriminant with respect to $\lambda_0$ is:
	$$
\mathcal D_{\lambda_0}F_2=16a^3b^3.
	$$
\end{example}

\begin{remark} We observe that in the Examples \ref{lemma:p2} and \ref{example:c5} we are getting disciminantly separable polynomials. In the rest of the Examples \ref{example:p3}, \ref{example:p4},
	\ref{example:p5}, \ref{example:c1}, \ref{example:c2}, \ref{example:c3}, \ref{example:c4}, the polynomials are not discriminantly separable, but discriminantly factorizable. However, it is important to note that the factors in all  these latter examples are homogeneous.  Thus, by a change of variables in the polynomials $p_3$, $p_4$, $p_5$ $(x, y)\mapsto (x, z)$, $z=x/y$ we are getting discriminantly separable polynomials in the new coordinates $(x, z)$. Similarly, in Examples \ref{example:c1}, \ref{example:c2}, \ref{example:c3}, \ref{example:c4}, the change of variables $(a, b)\mapsto (a, \hat b)$, with $\hat b=a/b$, transforms the polynomials into discriminantly separable polynomials in new variables $(a, \hat b)$.
	It would be very interesting to study further the observed relationship between $n$-valued groups and Cayley-type conditions with the discriminantly separable polynomials.
\end{remark}

\subsection{2-valued group  on $\mathbf {CP}^1$, disciminantly separable polynomials and Great Poncelet theorem for triangles}

\medskip

It appears that the
general equation of pencil of conics corresponds to an action of a two valued group.
We used this correspondence to provide a novel interpretation of  'the
mysterious Kowalevski change of variables'.
It turned out that the associativity condition for
this action  is equivalent to the
Great Poncelet Theorem for a triangle, as it was observed in \cite{Drag2010}. We are going to close the loop in
the paper with a brief reminder at the end about the Great Poncelet Theorem for triangles and this relationship to the associativity of the two-valued group, and thus, with the Kowalevski change of variables.

Consider the general pencil equation
$\mathcal F(s,x_1, x_2)=0$,
with $s$ being the parameter of the pencil, and $x_1$, $x_2$ the Darboux coordinates (see \cite{Drag2010}).
That pencil is related to  two  elliptic curves:
$\tilde\Gamma_1: y^2=P(x)$ and 
$\tilde\Gamma_2: t^2=J(s),$
where the polynomials $P$, $J$ are of degrees four and three respectively. These two curves appear to be isomorphic.
Rewrite  the cubic one $\tilde\Gamma_2$ in the canonical
form $\tilde\Gamma_2: t^2=J'(s)=4s^3-g_2s-g_3$.
Let $\psi: \tilde\Gamma_2\rightarrow \tilde\Gamma_1$ be a
birational morphism between the curves induced by a
fractional-linear transformation $\hat \psi$ which maps the three zeros
of $J'$ and $\infty$ to the four zeros of the polynomial $P$.

The curve $\tilde\Gamma_2$ as a cubic has a group structure with the neutral element at infinity.
With the subgroup $\mathbf {Z}_2$, it defines the standard
two-valued group structure  on $\mathbf {CP}^1$ (see
\cite{Buch1990}):
\begin{equation}\label{eq:G2Z2}
s_1 *_c s_2 =
\left[-s_1-s_2+\left(\frac{t_1-t_2}{2(s_1-s_2)}\right)^2,-s_1-s_2+\left(\frac{t_1+t_2}{2(s_1-s_2)}\right)^2\right],
\end{equation}
where $t_i=J'(s_i)$, $i=1,2.$
\medskip
\begin{theorem}\label{th:G2Z21}\cite{Drag2010}
	The general pencil equation after fractional-linear transformations
	$$
	\mathcal F(s, \hat \psi^{-1}(x_1), \hat \psi^{-1}(x_2))=0
	$$
	induces the two-valued coset group structure $(\tilde\Gamma_2, \mathbf
	Z_2)$ defined by \eqref{eq:G2Z2}.
\end{theorem}

A proof is given in \cite{Drag2010}.

The
geometric meaning of the pencil equation and algebraic structure of
the two valued group give together a connection observed in \cite{Drag2010} with the Great
Poncelet Theorem (\cite {Poncelet1822}, see also \cite{BergerGeometryII,DragRadn2011book}). We recall the  formulation of the
Great Poncelet Theorem for triangles in the form we need
below.
\medskip
\begin{theorem}[Great Poncelet Theorem for triangles \cite {Poncelet1822}]
	Let $\C_1$, $\C_2$, $\C_3$, $\C$ be conics from a pencil and 
	$a_1$, $a_2$, $a_3$ tangent lines to $\C$, such that $a_1$, $a_2$
	intersect on $\C_1$, $a_2$, $a_3$ intersect on $\C_2$ and $a_1$, $a_3$
	intersect on $\C_3$. Moreover, we suppose that the tangents to $\C_1$, $\C_2$, $\C_3$ at the intersection points  are not
	concurrent. 
	Suppose that $b_1$, $b_2$ are tangents to $\C$ which
	intersect on $\C_1$. Then there exists $b_3$, a tangent to 
	$\C$ such that the triplet $(b_1,b_2,b_3)$ satisfies all conditions
	as $(a_1,a_2,a_3)$.
\end{theorem}
\medskip
The associativity condition for the
action of the two-valued group $(\Gamma_2,\mathbf Z_2)$ is as follows.
\medskip
\begin{theorem}\label{th:G2Z23}\cite{Drag2010}
	Associativity conditions for the group structure of the two-valued
	coset group $(\Gamma_2,\mathbf Z_2)$ and for its action on
	$\mathbf {CP}^1$ are equivalent to the great Poncelet theorem for a
	triangle.
\end{theorem}

The proof is given in \cite{Drag2010}.

\phantomsection 
\addcontentsline{toc}{section}{Acknowledgement} 

\subsection*{Acknowledgement}
This research  was supported
by
 the Serbian Ministry of Education, Science, and Technological Development,
Project 174020 \emph{Geometry and Topology of Manifolds, Classical
	Mechanics, and Integrable Dynamical Systems}; and
 the Australian Research Council, Project DP190101838 \emph{Billiards within confocal quadrics and beyond}.

The authors are grateful to Andrey Mironov for inspiring discussions during GDIS 2018.


\begin{bibdiv}
%\phantomsection 
\addcontentsline{toc}{section}{References} 

\begin{biblist}

\bib{Akh1}{article}{
	author={Akhiezer, N. I.},
	title={\"Uber einige Funktionen, welche in zwei gegebenen Interwallen am wenigsten von Null abweichen. I Teil},
	journal={Izvestiya Akad. Nauk SSSA, VII ser., Otd. mat. est. nauk},
	date={1932},
	number={9},
	pages={1163--1202}
}  

\bib{Akh2}{article}{
	author={Akhiezer, N. I.},
	title={\"Uber einige Funktionen, welche in zwei gegebenen Interwallen am wenigsten von Null abweichen. II Teil},
	journal={Izvestiya Akad. Nauk SSSA, VII ser., Otd. mat. est. nauk},
	date={1933},
	number={3},
	pages={309--344}
}  

\bib{Akh3}{article}{
	author={Akhiezer, N. I.},
	title={\"Uber einige Funktionen, welche in zwei gegebenen Interwallen am wenigsten von Null abweichen. III Teil},
	journal={Izvestiya Akad. Nauk SSSA, VII ser., Otd. mat. est. nauk},
	date={1933},
	number={4},
	pages={499--536}
}  



\bib{AhiezerAPPROX}{book}{
	author={Ahiezer, N. I.},
	title={Lekcii po Teorii Approksimacii},
	language={Russian},
	publisher={OGIZ, Moscow-Leningrad},
	date={1947},
	pages={323}
}



\bib{Akh4}{book}{
	author={Akhiezer, N. I.},
	title={Elements of the theory of elliptic functions},
	series={Translations of Mathematical Monographs},
	volume={79},
	note={Translated from the second Russian edition by H. H. McFaden},
	publisher={American Mathematical Society, Providence, RI},
	date={1990}
}




\bib{BergerGeometryII}{book}{
	author={Berger, Marcel},
	title={Geometry. II},
	series={Universitext},
	publisher={Springer-Verlag},
	place={Berlin},
	date={1987}
}






\bib{Buch1990}{article}{
	author={Bukhshtaber, V. M.},
	title={Functional equations that are associated with addition theorems
		for elliptic functions, and two-valued algebraic groups},
	language={Russian},
	journal={Uspekhi Mat. Nauk},
	volume={45},
	date={1990},
	number={3(273)},
	pages={185--186},
	translation={
		journal={Russian Math. Surveys},
		volume={45},
		date={1990},
		number={3},
		pages={213--215},
	},
}

\bib{Buch2006}{article}{
	author={Buchstaber, V.},    
	title={$n$-valued groups: theory and applications},    
	journal={Moscow Mathematical Journal},    
	volume={6},    
	number={1},    
	pages={57--84},    
	date={2006}
}



\bib{Cayley1853}{article}{
	author={Cayley, Arthur},
	title={Note on the porism of the in-and-circumscribed polygon},
	journal={Philosophical magazine},
	volume={6},
	date={1853},
	pages={99--102}
}





\bib{DarbouxSUR}{book}{
	author={Darboux, Gaston},
	title={
		Le\c{c}ons sur la th\'eorie
		g\'en\'erale des surfaces et les
		applications g\'eo\-m\'etri\-ques du
		calcul infinitesimal
	},
	publisher={Gauthier-Villars},
	address={Paris},
	date={1914},
	volume={2 and 3}
}





\bib{Drag2002}{article}{
	author={Dragovi\'c, Vladimir},
	title={The Appell hypergeometric functions and classical separable mechanical systems},
	journal={J. Phys. A: Math. Gen.},
	volume={35},
	number={9},
	date={2002},
	pages={2213--2221}
}





\bib{Drag2010}{article}{
	author={Dragovi\'c, Vladimir},
	title={Geometrization and generalization of the Kowalevski top},
	journal={Communications in Mathematical Physics},
	eprint={arXiv:0912.3027},
	date={2010},
	volume={298},
	number={1},
	pages={37--64},
	note={DOI: 10.1007/s00220-010-1066-z}
}


\bib{Drag2012}{article}{
	author={Dragovi\'c, Vladimir},
	title={Algebro-geometric approach to the Yang-Baxter equation and related
		topics},
	journal={Publ. Inst. Math. (Beograd) (N.S.)},
	volume={91(105)},
	date={2012},
	pages={25--48},
	issn={0350-1302},
}



\bib{DragKuk2014jgm}{article}{
	author={Dragovi\'c, Vladimir},
	author={Kuki\'c, Katarina},
	title={Discriminantly separable polynomials and quad-equations},
	journal={J. Geom. Mech.},
	volume={6},
	date={2014},
	number={3},
	pages={319--333},
}

\bib{DragKuk2014rcd}{article}{
	author={Dragovi\'c, Vladimir},
	author={Kuki\'c, Katarina},
	title={Systems of Kowalevski type and discriminantly separable
		polynomials},
	journal={Regul. Chaotic Dyn.},
	volume={19},
	date={2014},
	number={2},
	pages={162--184},
	issn={1560-3547},
	review={\MR{3189255}},
	doi={10.1134/S1560354714020026},
}

\bib{DragKuk2014steklov}{article}{
	author={Dragovi\'c, Vladimir},
	author={Kuki\'c, Katarina},
	title={The Sokolov case, integrable Kirchhoff elasticae, and genus 2
		theta functions via discriminantly separable polynomials},
	journal={Proc. Steklov Inst. Math.},
	volume={286},
	date={2014},
	number={1},
	pages={224--239},
}

\bib{DragKuk2017}{article}{
	author={Dragovi\'c, Vladimir},
	author={Kuki\'c, Katarina},
	title={Discriminantly separable polynomials and generalized Kowalevski top},
	journal={Theoretical and Applied Mechanics},
	volume={44},
	date={2017},
	number={2},
	pages={229--236},
}




\bib{DragRadn2004}{article}{
	author={Dragovi\'c, Vladimir},
	author={Radnovi\'c, Milena},
	title={Cayley-type conditions for billiards within $k$ quadrics in $\mathbf R^d$},
	journal={J. of Phys. A: Math. Gen.},
	volume={37},
	pages={1269--1276},
	date={2004}
}



\bib{DragRadn2011book}{book}{
	author={Dragovi\'c, Vladimir},
	author={Radnovi\'c, Milena},
	title={Poncelet Porisms and Beyond},
	publisher={Springer Birkhauser},
	date={2011},
	place={Basel}
}



\bib{DragRadn2018}{article}{
	author={Dragovi{\'c}, Vladimir},
	author={Radnovi{\'c}, Milena},
	title={Periodic ellipsoidal billiard trajectories and extremal polynomials},
	date={2018},
	eprint={arXiv:1804.02515 [math.DS]}
}



\bib{DKN}{article}{
	author={Dubrovin, B. A.},
	author={Krichever, I. M.},
	author={Novikov, S. P.},
	title={Integrable systems. I},
	conference={
		title={Dynamical systems, IV},
	},
	book={
		series={Encyclopaedia Math. Sci.},
		volume={4},
		publisher={Springer, Berlin},
	},
	date={2001},
	pages={177--332}
}




\bib{DuistermaatBOOK}{book}{
	author={Duistermaat, Johannes J.},
	title={Discrete integrable systems: QRT maps and elliptic surfaces},
	series={Springer Monographs in Mathematics},
	publisher={Springer},
	place={New York},
	date={2010},
	pages={xxii+627},
	isbn={978-1-4419-7116-6}
}





\bib{Fed2001}{article}{
	author={Fedorov, Yuri},
	title={An ellipsoidal billiard with quadratic potential},
	journal={Funct. Anal. Appl.},
	volume={35},
	date={2001},
	number={3},
	pages={199--208}
}




\bib{GrifHar1977}{article}{
	author={Griffiths, Philip},
	author={Harris, Joe},
	title={A Poncelet theorem in space},
	journal={Comment. Math. Helvetici},
	volume={52},
	date={1977},
	number={2},
	pages={145--160}
}



\bib{GrifHar1978}{article}{
	author={Griffiths, Philip},
	author={Harris, Joe},
	title={On Cayley's explicit solution to Poncelet's porism},
	journal={EnsFeign. Math.},
	volume={24},
	date={1978},
	number={1-2},
	pages={31--40}
}




\bib{Hal1888}{book}{
	author={Halphen, G.-H.},
	title={Trait\' e des fonctiones elliptiques et de leures applications},
	part={deuxieme partie}
	publisher={Gauthier-Villars et fils},
	address={Paris},
	date={1888}
}






\bib{JacobiGW}{book}{
	author={Jacobi, Carl},
	title={Vorlesungen \"uber Dynamic. Gesammelte Werke, Supplementband},
	date={1884},
	address={Berlin}
}




\bib{KL1991}{article}{
	author={Korsch, H. J.},
	author={Lang, J.},
	title={A new integrable gravitational billiard},
	journal={J. Phys. A},
	volume={24},
	date={1991},
	number={1},
	pages={45--52},
}

\bib{Kow1889}{article}{
	author={Kowalevski, S.},
	title={Sur le probl\`eme de la rotation d'un corp solide autour d'un point fixe},
	journal={Acta Math.},
	volume={12},
	date={1889},
	pages={177--232}
}



\bib{KozTrBIL}{book}{
	author={Kozlov, Valery},
	author={Treshch\"ev, Dmitry},
	title={Billiards},
	publisher={Amer. Math. Soc.},
	address={Providence RI},
	date={1991}
}



\bib{KLN1990}{article}{
	author={Kre\u\i n, M. G.},
	author={Levin, B. Ya.},
	author={Nudel\cprime man, A. A.},
	title={On special representations of polynomials that are positive on a
		system of closed intervals, and some applications},
	note={Translated from the Russian by Lev J. Leifman and Tatyana L.
		Leifman},
	conference={
		title={Functional analysis, optimization, and mathematical economics},
	},
	book={
		publisher={Oxford Univ. Press, New York},
	},
	date={1990},
	pages={56--114}
}



\bib{LebCONIQUES}{book}{
	author={Lebesgue, Henri},
	title={Les coniques},
	publisher={Gauthier-Villars},
	address={Paris},
	date={1942}
}




\bib{PS1999}{article}{
	author={Peherstorfer, F.},
	author={Schiefermayr, K.},
	title={Description of extremal polynomials on several intervals and their
		computation. I, II},
	journal={Acta Math. Hungar.},
	volume={83},
	date={1999},
	number={1-2},
	pages={27--58, 59--83}
}


\bib{Poncelet1822}{book}{
	author={Poncelet, Jean Victor},
	title={Trait\'e des propri\'et\'es projectives des figures},
	publisher={Mett},
	address={Paris},
	date={1822}
}

\bib{PT2011}{article}{
	author={Popov, G.},
	author={Topalov, P.},
	title={On the integral geometry of Liouville billiard tables},
	journal={Comm. Math. Phys.},
	volume={303},
	date={2011},
	number={3},
	pages={721--759}
}



\bib{Radn2015}{article}{
	author={Radnovi\'c, Milena},
	title={Topology of the elliptical billiard with the Hooke's potential},
	journal={Theoretical and Applied Mechanics},
	date={2015},
	volume={42},
	number={1},
	pages={1--9}
}


\bib{RR2014}{article}{
	author={Ram\'\i rez-Ros, Rafael},
	title={On Cayley conditions for billiards inside ellipsoids},
	journal={Nonlinearity},
	volume={27},
	date={2014},
	number={5},
	pages={1003--1028}
}


\end{biblist}
\end{bibdiv}
\end{document}